\documentclass{amsart}
\usepackage{amsmath}
\usepackage{amssymb}
\usepackage{amsthm}

\DeclareMathOperator{\tr}{tr}

\DeclareMathOperator{\dvol}{dvol}

\DeclareMathOperator{\Ric}{Ric}

\DeclareMathOperator{\Rm}{Rm}

\DeclareMathOperator{\End}{End}

\DeclareMathOperator{\Sym}{Sym}

\DeclareMathOperator{\hash}{\sharp}

\DeclareMathOperator{\Met}{Met}

\newcommand{\omu}{\overline{\mu}}

\newcommand{\cB}{\widetilde{B}}
\newcommand{\cD}{\widetilde{D}}
\newcommand{\cE}{\widetilde{E}}
\newcommand{\cR}{\widetilde{R}}
\newcommand{\cS}{\widetilde{S}}
\newcommand{\cT}{\widetilde{T}}
\newcommand{\cY}{\widetilde{Y}}

\newcommand{\cphi}{\widetilde{\phi}}
\newcommand{\csigma}{\widetilde{\sigma}}

\newcommand{\hE}{\widehat{E}}
\newcommand{\hT}{\widehat{T}}
\newcommand{\hmW}{\widehat{\mW}}
\newcommand{\hGamma}{\widehat{\Gamma}}
\newcommand{\hsigma}{\widehat{\sigma}}
\newcommand{\lp}{\langle}
\newcommand{\rp}{\rangle}
\newcommand{\lv}{\lvert}
\newcommand{\rv}{\rvert}

\newcommand{\mC}{\mathcal{C}}

\newcommand{\mF}{\mathcal{F}}

\newcommand{\mS}{\mathcal{S}}

\newcommand{\mW}{\mathcal{W}}

\newcommand{\bN}{\mathbb{N}}

\newcommand{\bR}{\mathbb{R}}

\newcommand{\cRic}{\widetilde{\Ric}}
\def\sideremark#1{\ifvmode\leavevmode\fi\vadjust{\vbox to0pt{\vss
 \hbox to 0pt{\hskip\hsize\hskip1em
 \vbox{\hsize3cm\tiny\raggedright\pretolerance10000
 \noindent #1\hfill}\hss}\vbox to8pt{\vfil}\vss}}}

\newcommand{\comment}[1]{}

\newtheorem{thm}{Theorem}[section]
\newtheorem{prop}[thm]{Proposition}
\newtheorem{lem}[thm]{Lemma}
\newtheorem{cor}[thm]{Corollary}

\theoremstyle{definition}
\newtheorem{defn}[thm]{Definition}

\theoremstyle{remark}
\newtheorem{remark}[thm]{Remark}

\numberwithin{equation}{section}

\begin{document}

\title[The weighted $\sigma_k$-curvature of manifolds with density]{A notion of the weighted $\sigma_k$-curvature for manifolds with density}
\author{Jeffrey S. Case}
\thanks{Partially supported by NSF Grant DMS-1004394}
\address{Department of Mathematics \\ Princeton University \\ Princeton, NJ 08544}
\email{jscase@math.princeton.edu}
% \date{\today}
\begin{abstract}
We propose a natural definition of the weighted $\sigma_k$-curvature for a manifold with density; i.e.\ a triple $(M^n,g,e^{-\phi}\dvol)$.  This definition is intended to capture the key properties of the $\sigma_k$-curvatures in conformal geometry with the role of pointwise conformal changes of the metric replaced by pointwise changes of the measure.  We justify our definition through three main results.  First, we show that shrinking gradient Ricci solitons are local extrema of the total weighted $\sigma_k$-curvature functionals when the weighted $\sigma_k$-curvature is variational.  Second, we characterize the shrinking Gaussians as measures on Euclidean space in terms of the total weighted $\sigma_k$-curvature functionals.  Third, we characterize when the weighted $\sigma_k$-curvature is variational.  These results are all analogues of their conformal counterparts, and in the case $k=1$ recover some of the well-known properties of Perelman's $\mW$-functional.
\end{abstract}
\maketitle

%%%%%%%%%%%%%%%%%%%%%%%%%%%%%%%%%%%%%%%%%%%%%%%%%%%%%%%%%%%%%%%%%
%                                                               %
% Structure of the document                                     %
%                                                               %
% 1. Intro                                                      %
% *. Acknowledgments                                            %
% 2. Algebraic Preliminaries                                    %
% 3. Smooth Metric Measure Spaces                               %
% 4. The Weighted $\sigma_k$-curvature                          %
% 5. Variational Status                                         %
% 6. An Obata-type Theorem                                      %
% 7. First Variation of $\sigma_2$                              %
%                                                               %
%%%%%%%%%%%%%%%%%%%%%%%%%%%%%%%%%%%%%%%%%%%%%%%%%%%%%%%%%%%%%%%%%

\section{Introduction}
\label{sec:intro}

First introduced by Viaclovsky~\cite{Viaclovsky2000}, the $\sigma_k$-curvature is an important Riemannian invariant studied in conformal geometry which has many important connections to questions in analysis, topology and geometry.  For example, the study of the $\sigma_k$-curvature has led to the development of techniques for solving and characterizing solutions of fully-nonlinear second order PDE~\cite{CGS1989,ChangGurskyYang2002,LiLi2003,ShengTrudingerWang2007}, it has led to the development of new geometric functional inequalities~\cite{GuanWang2004,ShengTrudingerWang2007}, it is closely related to the Euler characteristic in even dimensions~\cite{Viaclovsky2000}, a fact which has been particularly fruitful in the study of four-manifolds~\cite{ChangGurskyYang2003,Gursky1998}, and it can be used to characterize spaceforms~\cite{GurskyViaclovsky2001} and give criteria for the existence of metrics with positive Ricci curvature~\cite{ChangGurskyYang2002,GuanViaclovskyWang2003,GurskyViaclovsky2001}.  In the special case $k=1$, the $\sigma_k$-curvature is a nonzero constant multiple of the scalar curvature, an object whose analytic, topological and geometric significance is much better understood; cf.\ \cite{GromovLawson1983,LeeParker1987,SchoenYau1979s}.

First introduced by Hamilton~\cite{Hamilton1982}, the Ricci flow is another important tool in Riemannian geometry with many analytic, topological, and geometric applications.  Perhaps the most famous is Perelman's resolution~\cite{Perelman1} of the Poincar\'e Conjecture, the proof of which relies heavily on all of these components.  In particular, we highlight the role of Perelman's $\mW$-functional
\[ \mW(g,\phi,\tau) := \int_M \left[ \tau(R+\lv\nabla\phi\rv^2)+\phi-n\right] (4\pi\tau)^{-\frac{n}{2}}e^{-\phi}\dvol . \]
Perelman showed that the $\mW$-functional and the $\nu$-entropy $\nu(g)$ --- defined by minimizing the $\mW$-functional over all pairs $(\phi,\tau)$ such that $\int (4\pi\tau)^{-\frac{n}{2}}e^{-\phi}\dvol=1$ --- are nondecreasing along the Ricci flow, and moreover, if $\nu(g)$ is finite, then $(M^n,g)$ is $\kappa$-noncollapsed~\cite{Perelman1}.  In this way the $\mW$-functional and the $\nu$-entropy are analogous to the total scalar curvature functional and the Yamabe constant, respectively.  Indeed, Perelman described (cf.\ Section~\ref{sec:smms}) multiple ways in which the $\mW$-functional should be regarded as the appropriate notion of the total scalar curvature functional on a manifold with density, while positivity of the Yamabe constant implies noncollapsing via uniform control on the $L^2$-Sobolev constant.  Here, a manifold with density is a triple $(M^n,g,e^{-\phi}\dvol)$ of a Riemannian manifold $(M^n,g)$ together with a smooth measure $e^{-\phi}\dvol$ determined by a function $\phi\in C^\infty(M)$ and the Riemannian volume element $\dvol$ determined by $g$.  The relationship between the $\mW$-functional and the Yamabe functional can be made more precise through the language of smooth metric measure spaces, wherein one attaches a notion of ``dimension'' to the measure $e^{-\phi}\dvol$ which interpolates between the Riemannian case (when this is the dimension of $M$) and the case where Perelman's $\mW$-functional arises (when this is infinite); one approach to this is presented in~\cite{Case2013y}.

Inspired by the successes of the $\sigma_k$-curvatures and the $\mW$-functional in Riemannian geometry, we both offer definitions of the weighted $\sigma_k$-curvatures on a manifold with density and their attendant $\mW_k$-functionals and establish some properties of these invariants.  It is our hope that the weighted $\sigma_k$-curvatures can be used to analytically study gradient Ricci solitons.  For this reason, this article focuses on showing that the weighted $\sigma_k$-curvatures have similar properties as those of the $\sigma_k$-curvatures as used to study Einstein metrics and certain geometric PDEs.

There are three properties of the $\sigma_k$-curvatures which we wish to capture.  First, the $\sigma_k$-curvatures are variational in a given conformal class if and only if $k\in\{1,2\}$ or the conformal class is locally conformally flat~\cite{BransonGover2008,Viaclovsky2000}.  Second, when the $\sigma_k$-curvature is variational, Einstein metrics are always critical points of the volume-normalized total $\sigma_k$-curvature functional~\cite{Viaclovsky2000}.  Moreover, Einstein metrics of nonzero scalar curvature are local extrema of this functional~\cite{Viaclovsky2000}.  Third, Einstein metrics are the only critical points in the conformal class of the standard sphere of the volume-normalized total $\sigma_k$-curvature functional in the positive elliptic $k$-cones~\cite{Viaclovsky2000} (cf.\ \cite{CGS1989,LiLi2003}).

As we shall see, the weighted $\sigma_k$-curvatures have the same properties provided one replaces ``locally conformally flat'' by ``flat'' --- that is, all sectional curvatures vanish --- and one replaces ``Einstein'' by ``gradient Ricci soliton.''  Explanations for why these are the right analogues are given in Section~\ref{sec:smms} (cf.\ \cite{Case2011t,Case2014s}).  To that end, we here discuss our results specifically in the case of the weighted $\sigma_1$- and $\sigma_2$-curvatures on manifolds with density, which are the two cases that we expect to be of the most interest.

Given a manifold with density $(M^n,g,e^{-\phi}\dvol)$ and a parameter $\lambda\in\bR$, we define the weighted $\sigma_1$- and $\sigma_2$-curvatures by
\begin{align*}
\csigma_{1,\phi} & := \frac{1}{2}\left( R + 2\Delta\phi - \lv\nabla\phi\rv^2 \right) + \lambda(\phi-n) , \\
\csigma_{2,\phi} & := \frac{1}{2}\left( (\csigma_{1,\phi})^2 - \lv\Ric + \nabla^2\phi - \lambda g\rv^2 \right) .
\end{align*}
We denote by $\sigma_{k,\phi}$ the weighted $\sigma_k$-curvature with $\lambda=0$; in particular, $\sigma_{1,\phi}$ is one-half the weighted scalar curvature used by Perelman~\cite{Perelman1}.  Note that, in an effort to keep our notation simple, we do not explicitly record the parameter $\lambda$ in our notation; instead, we always write $\csigma_{k,\phi}$ when $\lambda$ is possibly nonzero and specify its value in context.

We include the parameter $\lambda$ for two reasons.  First, one readily shows (cf.\ \cite{Hamilton_1995}) that if $(M^n,g,e^{-\phi}\dvol)$ is a gradient Ricci soliton with $\Ric+\nabla^2\phi=\lambda g$, then the weighted $\sigma_k$-curvatures are both constant.  Second, the $\mW_k$-functionals
\[ \mW_k(g,\phi,\tau) := \int_M \tau^k\csigma_{k,\phi}(4\pi\tau)^{-\frac{n}{2}}e^{-\phi}\dvol_g \]
for $\csigma_{k,\phi}$ defined in terms of $\lambda=\frac{1}{2\tau}$ have many nice properties.  For example, Perelman~\cite{Perelman1} showed that critical points of $\mW_1$ in
\[ \mC_1 := \left\{ (g,\phi,\tau) \colon \int_M (4\pi\tau)^{-\frac{n}{2}}e^{-\phi}\dvol_g = 1 \right\} \]
(resp.\ $\mC_1(g):=\{(\phi,\tau)\colon(g,\phi,\tau)\in\mC_1\}$) are shrinking gradient Ricci solitons (resp.\ have $\csigma_{1,\phi}$ constant).  We prove the following analogue for the $\mW_2$-functional.

\begin{thm}
\label{thm:intro/vary_w2}
Let $M^n$ be a compact manifold without boundary.
\begin{enumerate}
\item $(g,\phi,\tau)\in\mC_1$ is a critical point of $\mW_2\colon\mC_1\to\bR$ if and only if
\begin{equation}
\label{eqn:intro/vary_w2/einstein}
\delta_\phi d\cRic_\phi + \Rm\cdot\cRic_\phi + (\cRic_\phi)^2 - \csigma_{1,\phi}\cRic_\phi + \frac{1}{2\tau}\cRic_\phi = 0
\end{equation}
for $\cRic_\phi:=\Ric+\nabla^2\phi-\frac{1}{2\tau}$.
\item Fix a Riemannian metric $g$ on $M$.  Then $(\phi,\tau)\in\mC_1(g)$ is a critical point of $\mW_2\colon\mC_1(g)\to\bR$ if and only if there is a constant $c\in\bR$ such that
\begin{align}
\label{eqn:intro/vary_w2/csc} \csigma_{2,\phi} - \frac{1}{2\tau}\csigma_{1,\phi} & = c , \\
\label{eqn:intro/vary_w2/csc_trfree} \int_M \tr\left(\cB_\phi + \cE_{2,\phi} - \frac{1}{2\tau}\cE_{1,\phi}\right)\,e^{-\phi}\dvol_g & = 0
\end{align}
for $\cB_\phi + \cE_{2,\phi} - \frac{1}{2\tau}\cE_{1,\phi}$ the left-hand side of~\eqref{eqn:intro/vary_w2/einstein}.
\end{enumerate}
In particular, if $(M^n,g,e^{-\phi}\dvol)$ is a shrinking gradient Ricci soliton satisfying $\Ric+\nabla^2\phi=\frac{1}{2\tau}g$, then $(g,\phi,\tau)$ is a critical point of $\mW_2\colon\mC_1\to\bR$.
\end{thm}

See Section~\ref{sec:smms} for our conventions in defining the first two summands in~\eqref{eqn:intro/vary_w2/einstein} and Section~\ref{sec:defn} for additional properties of the $\mW_k$-functionals.  Though we shall not do so explicitly, our proofs are easily modified to prove similar statements relevant to steady (resp.\ expanding) gradient Ricci solitons by taking instead the parameter $\lambda=0$ (resp.\ $\lambda=-\frac{1}{2\tau}$); cf.\ \cite{FeldmanIlmanenNi2005,Perelman1}.

Note that we could combine Theorem~\ref{thm:intro/vary_w2} and Perelman's computations~\cite{Perelman1} of the first variation of the $\mW_1$-functional to find a functional whose critical points in $\mC_1(g)$ have $\csigma_{2,\phi}$ constant while preserving the property that shrinking gradient Ricci solitons are critical points within $\mC_1$.  We have instead opted to use the total weighted $\sigma_2$-curvature functional $\mW_2$ both because it more neatly parallels Perelman's definition of the $\mW$-functional and because in the case of smooth metric measure spaces there is essentially no choice as to which definition of a ``$\mW_2$-functional'' to use so that the analogues of gradient Ricci solitons are critical points; see~\cite{Case2014s} for details.  We expect a continued study of the weighted $\sigma_2$-curvature to shed additional insights into the relative merits of the functionals $\mW_2+t\mW_1$ for various values of $t$.

Perelman~\cite{Perelman1} showed that if $(M^n,g)$ is a compact shrinking gradient Ricci soliton with $(\phi,\tau)\in\mC_1(g)$ the corresponding critical point of the $\mW_1$-functional and if $(\phi_1,\tau_1)\in\mC_1(g)$ is a critical point of $\mW_1\colon\mC_1(g)\to\bR$, then $(\phi_1,\tau_1)=(\phi,\tau)$.  This result should be regarded as the weighted analogue of Obata's classification~\cite{Obata1971} of conformally Einstein constant scalar curvature metrics.  While a general Obata-type result for conformally Einstein metrics of constant $\sigma_k$-curvature is not known, it is known in the conformally flat setting~\cite{Viaclovsky2000}.  We prove the following weighted analogue.

\begin{thm}
 \label{thm:intro/obata/2}
 Suppose that $(\phi,\tau)\in C^\infty(\bR^n)\times\bR_+$ is a critical point of the functional $\mW_2\colon\mC_1(dx^2)\to\bR$.  Suppose additionally that
 \begin{equation}
  \label{eqn:intro/ellipticity/2}
  \begin{split}
   \csigma_{1,\phi} & < \frac{1}{2\tau}, \\
   \csigma_{2,\phi} & > \frac{1}{2\tau}\csigma_{1,\phi} - \frac{1}{8\tau^2},
  \end{split}
 \end{equation}
 and that the function $\cphi(x)=\lv x\rv^2\phi\left(x/\lv x\rv^2\right)$ extends to a $C^2$ function in $\bR^n$ such that $\cphi(0)>0$.  Then there is a point $x_0\in\bR^n$ such that
 \begin{equation}
  \label{eqn:intro/obata/2}
  \phi(x) = \frac{\lv x-x_0\rv^2}{4\tau} .
 \end{equation}
\end{thm}

The conditions~\eqref{eqn:intro/ellipticity/2} specify that $(\phi,\tau)$ is in the negative elliptic $2$-cone; see Section~\ref{sec:defn} for further discussion.  In this way, Theorem~\ref{thm:intro/obata/2} states that the only pairs $(\phi,\tau)\in\mC_1(dx^2)$ which are in the negative elliptic $2$-cone and are critical points of the $\mW_2$-functional are the shrinking Gaussians $\Bigl(\bR^n,dx^2,e^{-\frac{\lv x-x_0\rv^2}{4\tau}}\dvol\Bigr)$.

For any $k\in\bN$, it is known that compact Einstein metrics of nonzero scalar curvature (locally conformally flat if $k\geq3$) are local extrema of the volume-normalized total $\sigma_k$-curvature functionals, and moreover, they are strict local extrema unless the metric is the standard metric on the sphere~\cite{Viaclovsky2000}.  The analogue of this result for the weighted $\sigma_2$-curvature is as follows.

\begin{thm}
\label{thm:intro/local_extrema_W2}
Let $(M^n,g,e^{-\phi}\dvol)$ be a shrinking gradient Ricci soliton with $\Ric+\nabla^2\phi=\frac{1}{2\tau}g$ and $(\phi,\tau)\in\mC_1(g)$.  Let $\{(\phi_t,\tau_t)\}_{t\in(-\varepsilon,\varepsilon)}\subset\mC_1(g)$ be a smooth variation of $(\phi,\tau)$.  Then
\begin{equation}
\label{eqn:intro/local_extrema_W2}
\left.\frac{d^2}{dt^2}\mW_2\left(g,\phi_t,\tau_t\right)\right|_{t=0} \leq 0 .
\end{equation}
Moreover, equality holds in~\eqref{eqn:intro/local_extrema_W2} for nontrivial variations $(\phi_t,\tau_t)\in\mC_1(g)$ if and only if $(M^n,g,e^{-\phi}\dvol)$ factors as an isometric product with a shrinking Gaussian and the variation is tangent to a curve through shrinking Gaussians in the Euclidean factor.
\end{thm}

Theorem~\ref{thm:intro/vary_w2}, Theorem~\ref{thm:intro/obata/2}, and Theorem~\ref{thm:intro/local_extrema_W2} can be generalized to all $k$.  Theorem~\ref{thm:obata/k} is a direct generalization of Theorem~\ref{thm:intro/obata/2} which allows for arbitrary choices of the parameter $k$.  As the generalization of Theorem~\ref{thm:intro/vary_w2}, we show in Theorem~\ref{thm:variational_status} that the weighted $\sigma_k$-curvature is variational on a given Riemannian manifold if and only if $k\in\{1,2\}$ or the manifold is flat.  Moreover, we characterize the critical points of $\mW_k\colon\mC_1(g)\to\bR$ in these cases.  Theorem~\ref{thm:local_extrema_Wk} is a direct generalization of Theorem~\ref{thm:intro/local_extrema_W2} for the shrinking Gaussians.  Similar results can be formulated for steady and expanding gradient Ricci solitons; see Remark~\ref{rk:stability_steady_expanding}.

One of the next steps in realizing the potential of the weighted $\sigma_k$-curvatures as the ``right'' analogue of the $\sigma_k$-curvatures is to study the analytic problem of finding critical points of the $\mW_2$-functional and other related functionals.  Using work of Rothaus~\cite{Rothaus1981} and the Ricci flow, Perelman~\cite{Perelman1} constructed minimizers of the $\mW_1$-functional on any compact Riemannian manifold; an alternative approach which doesn't rely on Ricci flow is implicit in~\cite{Case2013y}.  One difficulty in extending this result to the $\mW_2$-functional is that its Euler--Lagrange equation~\eqref{eqn:intro/vary_w2/csc} for $\tau$ fixed is not always elliptic.  This can be introduced by introducing the appropriate notion of ``weighted elliptic cones,'' as is proven in Proposition~\ref{prop:elliptic_cone}.

The approach of this article is greatly inspired by previous work of the author~\cite{Case2013y} on the so-called weighted Yamabe problem on smooth metric measure spaces.  Indeed, one can similarly define the weighted $\sigma_k$-curvatures in that setting and obtain results similar those contained in this article.  This is carried out in~\cite{Case2014s}.  We have opted to present the case of the weighted $\sigma_k$-curvatures separately here for two reasons.  First, the treatment for smooth metric measure spaces is more technical, and having a separate treatment in the case of manifolds with density is helpful in motivating many of the arguments used in the more general case.  Second, the reduction from smooth metric measure spaces to manifolds with density involves taking limits in a way which is sometimes cumbersome and confusing.  The direct proofs given in this article are much more transparent.

Finally, we remark that a deficiency of the $\sigma_k$-curvatures as a tool to study Einstein metrics is that they are not always variational.  This problem can be overcome by the renormalized volume coefficients $v_k(g)$, which have the properties that they are always variational, that $v_k=\sigma_k$ for locally conformally flat metrics, that $v_k(e^{2\omega}g)$ depends only on the two-jet of $\omega$, and that Einstein metrics with nonzero scalar curvature are local extrema of the functional $g\mapsto\int v_k(g)\dvol_g$ within a conformal class of fixed-volume metrics~\cite{ChangFang2008,ChangFangGraham2012,Graham2009}.  Moreover, the renormalized volume coefficients can be regarded as perturbations of the $\sigma_k$-curvatures through lower order terms; see~\cite[Theorem~1.4]{Graham2009}.  We expect that similar ``weighted renormalized volume coefficients'' can be defined on manifolds with density.  Indeed, the author has defined the weighted analogue of $v_3$ by writing down a local formula and checking that it has the desired properties~\cite{Case2016v}.  A definition of the weighted renormalized volume coefficients of all orders seems to require the development of a suitable notion of weighted Poincar\'e metrics.

This article is organized as follows.  In Section~\ref{sec:polynomial} we discuss the main algebraic properties of what we call the weighted elementary symmetric polynomials, which formally can be regarded as the elementary symmetric polynomials on infinitely many variables, all but finitely many of which are the same.  In Section~\ref{sec:smms} we recall some basic ideas and definitions relevant to our study of manifolds with density.  In Section~\ref{sec:defn} we define the weighted $\sigma_k$-curvatures and their associated weighted Newton tensors and $\mW_k$-functionals.  In Section~\ref{sec:first_variation} we compute the first variation of the $\mW_1$- and $\mW_2$-functionals.  In Section~\ref{sec:divergence} we compute the divergence of the weighted Newton tensors, establishing in particular the conservation law for the weighted $\sigma_k$-curvatures on flat manifolds.  In Section~\ref{sec:obata} we prove our Obata theorems for the weighted $\sigma_k$-curvature.  In Section~\ref{sec:variational_status} we characterize when the weighted $\sigma_k$-curvature is variational.  In Section~\ref{sec:local_extrema} we show that shrinking gradient Ricci solitons are local extrema for the total weighted $\sigma_k$-curvature functionals when the weighted $\sigma_k$-curvature is variational.

% Acknowledgments
%\subsection*{Acknowledgments}

\section{Algebraic Preliminaries}
\label{sec:polynomial}

The purpose of this section is to discuss some of the basic algebraic properties of the following notion for ``weighted'' elementary symmetric polynomials.  In particular, we prove that they satisfy inequalities analogous to the usual Newton inequalities for the elementary symmetric polynomials (cf.\ \cite{HardyLittlewoodPolya1952}).

\begin{defn}
\label{defn:defn1}
Let $k\in\bN_0:=\bN\cup\{0\}$.  The \emph{$k$-th weighted elementary symmetric polynomial} $\sigma_k^\infty\colon\bR\times\bR^n\to\bR$ is recursively defined by $\sigma_0^\infty(\mu_0;\mu)=1$ and
\[ k\sigma_k^\infty(\mu_0;\mu) = \sigma_{k-1}^\infty(\mu_0;\mu)\sum_{j=0}^n\mu_j + \sum_{i=1}^{k-1}\sum_{j=1}^n (-1)^i\sigma_{k-1-i}^\infty(\mu_0;\mu)\mu_j^i \]
for $k\geq1$ and $\mu=(\mu_1,\dotsc,\mu_n)\in\bR^n$.
\end{defn}

Note that $\sigma_k^\infty(0;\mu)=\sigma_k(\mu)$ for $\sigma_k$ the usual $k$-th elementary symmetric polynomial on $n$-variables.  Clearly $\sigma_k^\infty(\mu_0;\mu)$ is independent of the ordering of the components of $\mu\in\bR^n$.

From a purely formal standpoint, the $k$-th weighted elementary symmetric polynomial can be understood as the limit
\[ \sigma_k^\infty(\mu_0;\mu) = \lim_{m\to\infty} \sigma_k\bigg(\mu,\underbrace{\frac{\mu_0}{m},\dotsc,\frac{\mu_0}{m}}_{\text{$m$ times}}\bigg) . \]
With this in mind, the algebraic properties of the weighted $k$-th elementary symmetric polynomial established in the remainder of this section should come as no surprise.

The first goal in this section is to prove that the weighted elementary symmetric polynomials satisfy the usual Newton inequalities when regarded as inequalities for the elementary symmetric polynomials on infinitely many variables.  In order to accomplish that goal, we need two equivalent formulations of the weighted elementary symmetric polynomials.  First, we find a relation between $\sigma_k^\infty(\mu_0;\mu)$ and $\sigma_k^\infty(\mu_1;\mu)$.  In particular, this formula shows that we can regard $\sigma_k^\infty(\mu_0;\mu)$ as a perturbation of $\sigma_k(\mu)$ through $\mu_0$ and the lower order terms $1,\sigma_1(\mu),\dotsc,\sigma_{k-1}(\mu)$.

\begin{prop}
\label{prop:ordinary_sigmak}
Given $k\in\bN_0$ and $\mu_0,\mu_1\in\bR$ and $\mu\in\bR^n$, it holds that
\begin{equation}
\label{eqn:new_ordinary_sigmak}
\sigma_k^\infty(\mu_0+\mu_1;\mu) = \sum_{j=0}^k \frac{\mu_1^j}{j!}\sigma_{k-j}^\infty(\mu_0;\mu) .
\end{equation}
In particular,
\begin{equation}
\label{eqn:ordinary_sigmak}
\sigma_k^\infty(\mu_1;\mu) = \sum_{j=0}^k \frac{\mu_1^j}{j!}\sigma_{k-j}(\mu) .
\end{equation}
\end{prop}

\begin{proof}

The proof is by strong induction.  Clearly~\eqref{eqn:new_ordinary_sigmak} is true when $k=0$.  Suppose it is known to be true for $j=0,\dotsc,k-1$.  For convenience, set $N_j=\sum_{i=1}^n\mu_i^j$, so that
\[ k\sigma_k^\infty(\mu_0+\mu_1;\mu) = (\mu_0+\mu_1)\sigma_{k-1}^\infty(\mu_0+\mu_1;\mu) + \sum_{j=0}^{k-1}(-1)^j\sigma_{k-1-j}^\infty(\mu_0+\mu_1;\mu) N_{j+1} . \]
By the inductive hypothesis we thus find that
\begin{align*}
k\sigma_k^\infty(\mu_0+\mu_1;\mu) & = \sum_{j=0}^{k-1}\frac{\mu_1^{j+1}}{j!}\sigma_{k-1-j}^\infty(\mu_0;\mu) + \sum_{j=0}^{k-1}\frac{\mu_1^j}{j!}\mu_0\sigma_{k-1-j}^\infty(\mu_0;\mu) \\
& \quad + \sum_{j=0}^{k-1}\sum_{\ell=0}^{k-1-j}(-1)^j\frac{\mu_1^\ell}{\ell!}\sigma_{k-1-j-\ell}^\infty(\mu_0;\mu)N_{j+1} \\
& = k\sum_{j=0}^k\frac{\mu_1^j}{j!}\sigma_{k-j}^\infty(\mu_0;\mu),
\end{align*}
where the second equality follows by reindexing the first summation, switching the order of the third summation, and using the definition of the weighted elementary symmetric polynomials.  The final claim~\eqref{eqn:ordinary_sigmak} follows by taking $\mu_0=0$ in~\eqref{eqn:new_ordinary_sigmak}.
\end{proof}

Second, we have the following generating function for the weighted elementary symmetric polynomials analogous to the usual generating function for the elementary symmetric polynomials, namely the case $\mu_0=0$ in~\eqref{eqn:generating_function} below.

\begin{prop}
\label{prop:generating_function}
For any $(\mu_0;\mu)\in\bR\times\bR^n$ and $t\in\bR$ it holds that
\begin{equation}
\label{eqn:generating_function}
e^{\mu_0t}\prod_{j=1}^n\left(1+\mu_jt\right) = \sum_{j=0}^\infty \sigma_j^\infty(\mu_0;\mu)t^j .
\end{equation}
\end{prop}

\begin{proof}

Combining the Maclaurin expansion for $e^{\mu_0t}$ with the formula
\[ \prod_{j=1}^n\left(1+\mu_jt\right) = \sum_{j=0}^n \sigma_j(\mu)t^j \]
yields the series expansion
\[ e^{\mu_0t}\prod_{j=1}^n\left(1+\mu_jt\right) = \sum_{j=0}^\infty\sum_{k=0}^j\frac{\mu_0^k}{k!}\sigma_{j-k}(\mu_0;\mu)t^j . \]
The desired result then follows immediately from~\eqref{eqn:ordinary_sigmak}.
\end{proof}

Using the generating function~\eqref{eqn:generating_function}, we can derive the weighted Newton inequalities for the weighted elementary symmetric polynomials by an argument analogous to the one usually given in the unweighted case (cf.\ \cite{HardyLittlewoodPolya1952}).

\begin{thm}
\label{thm:weighted_newton}
Given $k\geq1$ and $(\mu_0;\mu)\in\bR\times\bR^n$, it holds that
\begin{equation}
\label{eqn:weighted_newton}
\sigma_{k-1}^\infty(\mu_0;\mu)\sigma_{k+1}^\infty(\mu_0;\mu) \leq \frac{k}{k+1}\left(\sigma_k^\infty(\mu_0;\mu)\right)^2 .
\end{equation}
Moreover, equality holds in~\eqref{eqn:weighted_newton} if and only if $(\mu_0;\mu)$ satisfies
\begin{enumerate}
\item $\mu=0$, or
\item $\mu_0=0$ and, up to reindexing, it holds that $\mu_1=\dotsb=\mu_{n+1-k}=0$.
\end{enumerate}
\end{thm}

\begin{proof}

To begin, set $p_k^\infty=k!\sigma_k^\infty$, so that~\eqref{eqn:weighted_newton} is equivalently written
\begin{equation}
\label{eqn:weighted_newton_p}
p_{k-1}^\infty(\mu_0;\mu) p_{k+1}^\infty(\mu_0;\mu) \leq \left(p_k^\infty(\mu_0;\mu)\right)^2 .
\end{equation}

First, consider the case $\mu_0=0$.  Then $\sigma_k^\infty(\mu_0;\mu)=\sigma_k(\mu)$ and the result follows immediately from the usual Newton inequalities~\cite{HardyLittlewoodPolya1952}.

Next, consider next the case $k=1$.  A straightforward computation gives
\[ \left(p_1^\infty(\mu_0;\mu)\right)^2 - p_2^\infty(\mu_0;\mu) = \sum_{j=1}^n\mu_j^2, \]
from which the result immediately follows.

Finally, suppose $k\geq2$ and $\mu_0\not=0$.  By Proposition~\ref{prop:generating_function} we have that
\begin{equation}
\label{eqn:polynomial}
P(t) := e^{\mu_0t}\prod_{j=1}^n\left(1+\mu_jt\right) = \sum_{j=0}^\infty p_j^\infty(\mu_0;\mu)\frac{t^j}{j!} .
\end{equation}
Set $\ell=\#\{j\in\{1,\dotsc,n\}\colon \mu_j=0\}$.  Up to reindexing $\mu_1,\dotsc,\mu_n$, we may write $P(t)=e^{\mu_0t}Q_0(t)$ for $Q_0(t)$ a polynomial of degree $n-\ell$ with roots $r_j^{(0)}=-\mu_j^{-1}\not=0$ such that $r_1^{(0)}\leq\dotsb\leq r_{n-\ell}^{(0)}$.  Differentiating the right-hand side of~\eqref{eqn:polynomial} gives
\begin{equation}
\label{eqn:dk_polynomial}
\frac{d^{k-1}}{dt^{k-1}}P(t) = \sum_{j=0}^\infty p_{k+j-1}^\infty(\mu_0;\mu)\frac{t^j}{j!} .
\end{equation}
On the other hand, since $\mu_0\not=0$, repeated application of Rolle's Theorem implies that for $j\in\{1,\dotsc,k-1\}$ we may write
\begin{equation}
\label{eqn:rolle_polynomial}
\frac{d^j}{dt^j}P(t) = e^{\mu_0t}Q_j(t)
\end{equation}
for $Q_j(t)$ a polynomial of degree $n-\ell$ with roots $r_1^{(j)}\leq\dotsb\leq r_{n-\ell}^{(j)}$ satisfying
\begin{enumerate}
\item $r_1^{(j)}\leq r_1^{(j-1)} \leq r_2^{(j)} \leq \dotsb \leq r_{n-\ell}^{(j-1)}$ if $\mu_0>0$, and
\item $r_1^{(j-1)}\leq r_1^{(j)} \leq r_2^{(j-1)} \leq \dotsb \leq r_{n-\ell}^{(j)}$ if $\mu_0<0$ .
\end{enumerate}
In particular, for each $j\in\{1,\dotsc,k-1\}$ there can be at most one $i\in\{1,\dotsc,n-\ell\}$ such that $r_i^{(j)}=0$.  We consider these cases separately.

\textit{Case 1: There exists an $i\in\{1,\dotsc,n-\ell\}$ such that $r_i^{(k-1)}=0$}.  Thus $0$ is a zero of order one of $Q_{k-1}(t)$.  Comparing~\eqref{eqn:dk_polynomial} and~\eqref{eqn:rolle_polynomial} thus gives $p_{k-1}^\infty(\mu_0;\mu)=0$ but $p_k^\infty(\mu_0;\mu)\not=0$, and hence~\eqref{eqn:weighted_newton_p} holds with strict inequality.

\textit{Case 2: $r_i^{(k-1)}\not=0$ for all $i\in\{1,\dotsc,n-\ell\}$}.  Set $\mu_i^{(k-1)}=-1/r_i^{(k-1)}$ for $i\in\{1,\dotsc,n-\ell\}$ and $\mu_s^{(k-1)}=0$ for $s\in\{n-\ell+1,\dotsc,n\}$.  Since $0$ is not a root of $Q_{k-1}(t)$, we have that $p_{k-1}^\infty(\mu_0;\mu)\not=0$, and hence~\eqref{eqn:dk_polynomial} and~\eqref{eqn:rolle_polynomial} together imply that
\[ e^{\mu_0t}\prod_{j=1}^n\left(1+\mu_j^{(k-1)}t\right) = \sum_{j=0}^\infty\frac{p_{k+j-1}^\infty(\mu_0;\mu)}{p_{k-1}^\infty(\mu_0;\mu)}\frac{t^j}{j!} . \]
In particular, $\tilde p_j^\infty:=p_j^\infty(\mu_0;\mu_1^{(k-1)},\dotsc,\mu_n^{(k-1)})=\frac{p_{k+j-1}^\infty(\mu_0;\mu_1,\dotsc,\mu_n)}{p_{k-1}^\infty(\mu_0;\mu_1,\dotsc,\mu_n)}$.  We have already seen that $\left(\tilde p_1^\infty\right)^2\geq\tilde p_2^\infty$ with equality if and only if $\mu_1^{(k-1)}=\dotsb=\mu_n^{(k-1)}=0$, from which the result immediately follows.
\end{proof}

Like the Newton inequalities, the weighted Newton inequalities can be used to prove many interesting and useful inequalities.  For us, the most important inequality is~\eqref{eqn:newton_cor} below, which is valid only in the negative weighted elliptic $k$-cones.

\begin{defn}
Let $k\in\bN$.  The \emph{negative weighted elliptic $k$-cone $\Gamma_k^{\infty,-}$} is the set
\[ \Gamma_k^{\infty,-} = \left\{ (\mu_0;\mu)\in\bR\times\bR^n \colon (-1)^j\sigma_j^\infty(\mu_0;\mu)>0 \text{ for all $j\in\{1,\dotsc,k\}$} \right\} . \]
\end{defn}

By reason of our intended applications, we only consider the negative weighted elliptic $k$-cones.  However, one could also study the positive weighted elliptic $k$-cones, defined by requiring all intermediate weighted elementary symmetric polynomials to be positive, and in this case all of the results of this section still hold, up to some changes of signs (cf.\ \cite{CaffarelliNirenbergSpruck1985,Viaclovsky2000}).

Using the negative weighted elliptic $k$-cones we can establish the following analogue of~\cite[Lemma~23]{Viaclovsky2000}.

\begin{cor}
\label{cor:newton}
Let $k\in\bN$ and let $(\mu_0;\mu)\in\Gamma_k^{\infty,-}$.  Then
\begin{equation}
\label{eqn:newton_cor}
(-1)^{k+1}(k+1)\sigma_{k+1}^\infty(\mu_0;\mu) \leq (-1)^{k+1}\sigma_1^\infty(\mu_0;\mu)\sigma_k^\infty(\mu_0;\mu)
\end{equation}
and moreover, equality holds if and only if $\mu=0$.
\end{cor}

\begin{proof}

The proof is by induction.  Using the notation $p_k^\infty$ used in the proof of Theorem~\ref{thm:weighted_newton}, we see that the claim~\eqref{eqn:newton_cor} is equivalent to
\begin{equation}
\label{eqn:newton_cor_p}
(-1)^{k+1}p_{k+1}^\infty(\mu_0;\mu) \leq (-1)^{k+1}p_1^\infty(\mu_0;\mu) p_k^\infty(\mu_0;\mu) .
\end{equation}
It is clear from Theorem~\ref{thm:weighted_newton} that~\eqref{eqn:newton_cor} and the characterization of equality hold when $k=1$.  Suppose then that we know that
\[ (-1)^jp_j^\infty(\mu_0;\mu)\leq (-1)^j p_1^\infty(\mu_0;\mu) p_{j-1}^\infty(\mu_0;\mu) \]
for some $j\in\{2,\dotsc,k\}$ and that equality holds if and only if $\mu=0$.  Since $(\mu_0;\mu)\in\Gamma_k^{\infty,-}$, the inductive hypothesis and~\eqref{eqn:weighted_newton_p} together imply that
\[ p_{j-1}^\infty(\mu_0;\mu) p_{j+1}^\infty(\mu_0;\mu) \leq \left(p_j^\infty(\mu_0;\mu)\right)^2 \leq p_1^\infty(\mu_0;\mu) p_{j-1}^\infty(\mu_0;\mu) p_j^\infty(\mu_0;\mu) . \]
Since $(-1)^{j+1}p_{j-1}^\infty(\mu_0;\mu)>0$ by assumption,
\[ (-1)^{j+1}p_{j+1}^\infty(\mu_0;\mu) \leq (-1)^{j+1}p_1^\infty(\mu_0;\mu) p_j^\infty(\mu_0;\mu) \]
holds.  Moreover, if equality holds in the above display, then, since $(-1)^jp_j^\infty(\mu_0;\mu)>0$, it must also hold in the inductive step, and hence $\mu=0$.
\end{proof}

\subsection{Weighted Newton transformations}
\label{subsec:polynomial/transformation}

We can also consider weighted elementary symmetric polynomials acting on pairs of a real number and a symmetric matrix.  Since symmetric matrices are diagonalizable and the weighted elementary symmetric polynomials are invariant under permutations of the $\bR^n$-factor, these are naturally defined in terms of the eigenvalues of the matrix.

\begin{defn}
Let $P$ be a symmetric $n\times n$ real-valued matrix and let $\mu_0\in\bR$.  The \emph{$k$-th weighted elementary symmetric function $\sigma_k^\infty(\mu_0;P)$ of $P$ and $\mu_0$} is
\[ \sigma_k^\infty(\mu_0;P) = \sigma_k^\infty(\mu_0;\mu) \]
for $\mu=(\mu_1,\dotsc,\mu_n)$ the eigenvalues of $P$.
\end{defn}

We can also discuss the negative weighted elliptic cones in this setting, which we shall do with an abuse of notation.

\begin{defn}
\label{defn:endomorphism_elliptic_cone}
Let $k\in\bN_0$.  The \emph{negative weighted elliptic $k$-cone $\Gamma_k^{\infty,-}$} is the set
\[ \Gamma_k^{\infty,-} := \left\{ (\mu_0;P)\in \bR\times\Sym(\bR^n) \colon (\mu_0;\mu)\in \Gamma_k^{\infty,-} \right\} , \]
where $\Sym(\bR^n)$ is the space of symmetric $n\times n$ real-valued matrices and $\mu=(\mu_1,\dotsc,\mu_n)$ are the eigenvalues of $P$.
\end{defn}

As a generalization of the Newton transformations, for each $k\in\bN_0$ one can construct a new matrix $T_k^\infty$ which is intimately related to the weighted elementary symmetric function $\sigma_k^\infty$.

\begin{defn}
Let $P$ be a symmetric $n\times n$ real-valued matrix and let $\mu_0\in\bR$.  The \emph{$k$-th weighted Newton transformation $T_k^\infty(\mu_0;P)$ of $P$ and $\mu_0$} is
\[ T_k^\infty(\mu_0;P) = \sum_{j=0}^k (-1)^j\sigma_{k-j}^\infty(\mu_0;P)\,P^j . \]
\end{defn}

When the symmetric matrix $P$ and the scalar $\mu_0$ are clear from context, we shall simply use the symbols $\sigma_k^\infty$ and $T_k^\infty$ to denote the $k$-th weighted elementary symmetric function $\sigma_k^\infty(\mu_0;P)$ and the $k$-th weighted Newton transformation $T_k^\infty(\mu_0;P)$, respectively.

As we shall see in the remainder of this article, the weighted Newton transformations appear in many important ways when studying the weighted elementary symmetric functions.  For now, we only wish to point out that the definition of the weighted elementary symmetric polynomials can be rewritten in terms of the weighted Newton transformations in the following way.

\begin{lem}
\label{lem:contr_newton}
Let $P$ be a symmetric $n\times n$ real-valued matrix and let $\mu_0\in\bR$.  Then for all $k\in\bN_0$ it holds that
\[ \lp T_k^\infty, P\rp = (k+1)\sigma_{k+1}^\infty - \mu_0\sigma_k^\infty , \]
for $\lp A,B\rp := \tr AB$.
\end{lem}

\begin{proof}

This follows immediately from Definition~\ref{defn:defn1} and the fact
\[ \lp P^j, P\rp = \tr P^{j+1} = \sum_{i=1}^n \mu_i^{j+1} \]
for $(\mu_1,\dotsc,\mu_n)$ the eigenvalues of $P$.
\end{proof}

\subsection{The elliptic cones}
\label{subsec:polynomial/cones}

As in the Riemannian setting, the main reason for restricting our attention to the elliptic cones $\Gamma_k^{\infty,-}$ is that the problem of prescribing the weighted $\sigma_k$-curvature is elliptic in these cones; see~\cite{CaffarelliNirenbergSpruck1985,Viaclovsky2000} for the Riemannian statement and Proposition~\ref{prop:elliptic_cone} for the weighted statement.  The main point is that the symbol of the weighted $\sigma_k$-curvatures is given by the $(k-1)$-th Newton tensor.

In this subsection we establish the main algebraic ingredient behind this fact, namely that if $(\mu_0;\mu)\in\Gamma_k^{\infty,-}$, then for any symmetric matrix $P$ with eigenvalues given by the coordinates of $\mu$, the corresponding weighted Newton transformation $T_{k-1}^\infty(\mu_0;P)$ is positive (resp.\ negative) definite when $k$ is odd (resp.\ even).  We shall prove this directly, rather than arguing via hyperbolic polynomials as in~\cite{CaffarelliNirenbergSpruck1985}.  To that end, we require two lemmas.  First, we have the following relationship between the weighted elementary symmetric polynomials when removing one of the variables.

\begin{lem}
\label{lem:remove_one}
Let $(\mu_0;\mu)\in\bR\times\bR^n$ and fix $i\in\{1,2,\dotsc,n\}$.  Set
\[ \omu(i)=(\mu_1,\mu_2,\dotsc,\mu_{i-1},\mu_{i+1},\mu_{i+2},\dotsc,\mu_n); \]
i.e.\ $\omu$ is $\mu$ with the $i$-th coordinate removed.  Then for all $k\in\bN$ it holds that
\begin{equation}
\label{eqn:remove_one}
\sigma_k^\infty\left(\mu_0;\mu\right) = \sigma_k^\infty\left(\mu_0;\omu(i)\right) + \mu_i\sigma_{k-1}^\infty\left(\mu_0;\omu(i)\right) .
\end{equation}
\end{lem}

\begin{proof}

It is readily checked that $\sigma_k(\mu)=\sigma_k\left(\omu(i)\right)+\mu_i\sigma_k\left(\omu(i)\right)$.  The result is then an immediate consequence of Proposition~\ref{prop:ordinary_sigmak}.
\end{proof}

Second, we have the following identification of the eigenvalues of the weighted Newton transformations.

\begin{lem}
\label{lem:newton_eigenvalue}
Let $P$ be an $n\times n$ symmetric matrix with eigenvalues $\mu_1,\mu_2,\dotsc,\mu_n$ and let $\mu_0\in\bR$.  Then for any $i\in\{1,2,\dotsc,n\}$ and any $k\in\bN_0$, the $i$-th eigenvalue of the $k$-th weighted Newton transformation $T_k^\infty(\mu_0;P)$ is equal to $\sigma_k^\infty\left(\mu_0;\omu(i)\right)$.
\end{lem}

\begin{proof}

The proof is by induction.  Clearly the result is true if $k=0$.  Suppose then that the $i$-th eigenvalue of $T_{k-1}^\infty(\mu_0;P)$ is equal to $\sigma_{k-1}^\infty\left(\omu(i)\right)$.  Since $T_k^\infty=\sigma_k^\infty I - T_{k-1}^\infty P$, the $i$-th eigenvalue of $T_k^\infty$ is $\sigma_k^\infty(\mu_0;\mu)-\mu_i\sigma_{k-1}^\infty\left(\mu_0;\omu(i)\right)$.  The result is then an immediate consequence of Lemma~\ref{lem:remove_one}.
\end{proof}

Using these lemmas we can establish the aforementioned fact about the definiteness of the weighted Newton transformations associated to pairs $(\mu_0;P)$ in the negative weighted elliptic cones.

\begin{prop}
\label{prop:ellipticity}
Let $(\mu_0;P)\in\Gamma_{k+1}^{\infty,-}$.  Then $(-1)^kT_k^\infty(\mu_0;P)>0$.
\end{prop}

\begin{proof}

Let $\mu\in\bR^n$ denote the eigenvalues of $P$.  By Lemma~\ref{lem:newton_eigenvalue}, it suffices to show that $(-1)^k\sigma_k^\infty\left(\mu_0;\omu(i)\right)>0$ for all $i\in\{1,2,\dotsc,n\}$.  We do so by strong induction.  Clearly $\sigma_0^\infty\left(\mu_0;\omu(i)\right)>0$.  Suppose that $(-1)^\ell\sigma_\ell^\infty\left(\mu_0;\omu(i)\right)>0$ for $\ell\in\{0,1,\dotsc,j\}$ and some $j\leq k-1$.  From Lemma~\ref{lem:remove_one} and the assumption $(\mu_0;\mu)\in\Gamma_{k+1}^{\infty,-}$ we conclude that
\[ (-1)^{j}\left[\sigma_{j+2}^\infty\left(\mu_0;\omu(i)\right) + \mu_i\sigma_{j+1}^\infty\left(\mu_0;\omu(i)\right)\right] > 0 . \]
Using the inductive hypothesis, Theorem~\ref{thm:weighted_newton}, and Lemma~\ref{lem:remove_one}, we see that
\begin{align*}
0 & < \sigma_{j}^\infty\left(\mu_0;\omu(i)\right)\sigma_{j+2}^\infty\left(\mu_0;\omu(i)\right) + \mu_i\sigma_{j}^\infty\left(\mu_0;\omu(i)\right)\sigma_{j+1}^\infty\left(\mu_0;\omu(i)\right) \\
& \leq \sigma_{j+1}^\infty\left(\mu_0;\omu(i)\right)\left[ \sigma_{j+1}^\infty\left(\mu_0;\omu(i)\right) + \mu_i\sigma_{j}^\infty\left(\mu_0;\omu(i)\right)\right] \\
& = \sigma_{j+1}^\infty\left(\mu_0;\omu(i)\right)\sigma_{j+1}^\infty\left(\mu_0;\mu\right) .
\end{align*}
Since $(\mu_0;\mu)\in\Gamma_{k+1}^{\infty,-}$ we see that $(-1)^{j+1}\sigma_{j+1}^\infty\left(\mu_0;\omu(i)\right)>0$, as desired.
\end{proof}
\section{Manifolds with density}
\label{sec:smms}

In this section we collect some basic facts and definitions about manifolds with density which are needed in this article.  A \emph{manifold with density} is a triple $(M^n,g,e^{-\phi}\dvol)$ of a Riemannian manifold together with measure $e^{-\phi}\dvol$ determined by a function $\phi\in C^\infty(M)$ and the Riemannian volume element $\dvol$ of $g$.  These are sometimes also referred to as smooth metric measure spaces, but we have opted to use the terminology ``manifold with density'' so as to avoid confusion with the more general notion of a smooth metric measure space used in~\cite{Case2014s} which incorporates also a notion of the ``dimension'' of the measure.

The basic objects of study on manifolds with density are the \emph{weighted Laplacian} $\Delta_\phi$ and the \emph{Bakry-\'Emery Ricci tensor} $\Ric_\phi$.  The weighted Laplacian is defined by
\[ \Delta_\phi u = \Delta u - \lp\nabla u,\nabla\phi\rp \]
for all $u\in C^2(M)$, where $\Delta$ is the Laplacian with respect to $g$.  Equivalently, the weighted Laplacian is the Euler--Lagrange operator for the Dirichlet energy $\frac{1}{2}\int\lv\nabla u\rv^2e^{-\phi}\dvol$.  The Bakry-\'Emery Ricci tensor is defined by
\[ \Ric_\phi := \Ric + \nabla^2\phi , \]
where $\Ric$ is the Ricci curvature of $g$ and $\nabla^2\phi$ is the Hessian of $\phi$.  Equivalently, the Bakry-\'Emery Ricci tensor is the curvature term in the weighted Bochner formula, which states that
\[ \frac{1}{2}\Delta_\phi\lv\nabla u\rv^2 = \lv\nabla^2 u\rv^2 + \lp\nabla\Delta_\phi u,\nabla u\rp + \Ric_\phi(\nabla u,\nabla u) \]
for all $u\in C^\infty(M)$.  These objects form the basis for comparison theory on manifolds with density; see, for example, \cite{Wei_Wylie} and references therein.

In his study of the Ricci flow, Perelman~\cite{Perelman1} introduced the notion of the \emph{weighted scalar curvature} $R_\phi$ (see also~\cite{Lott2003}) as the scalar function
\[ R_\phi := R + 2\Delta\phi - \lv\nabla\phi\rv^2 . \]
While the weighted scalar curvature is not the trace of the Bakry-\'Emery Ricci tensor, it has two important properties which justify its name.  First, it is related to the Bakry-\'Emery Ricci tensor via the Bianchi identity $\delta_\phi\Ric_\phi=\frac{1}{2}dR_\phi$, where $\delta_\phi=\delta-\iota_{\nabla\phi}$ is the weighted divergence, so called because it is the (negative of the) formal adjoint of the Levi-Civita connection on one-forms when taken with respect to the measure $e^{-\phi}\dvol$.  Second, the weighted scalar curvature is the curvature term in the Weitzenb\"ock formula for the Dirac operator on spinors when all formal adjoints are taken with respect to $e^{-\phi}\dvol$.

In fact, Perelman's study of the Ricci flow suggests that there is another way one might define the Bakry-\'Emery Ricci tensor and the weighted scalar curvature.  Given a manifold with density $(M^n,g,e^{-\phi}\dvol)$ and a parameter $\lambda\in\bR$, we define the \emph{modified Bakry-\'Emery Ricci tensor} $\cRic_\phi$ and the \emph{modified weighted scalar curvature} $\cR_\phi$ by
\begin{align*}
\cRic_\phi & = \Ric + \nabla^2\phi - \lambda g, \\
\cR_\phi & = R + 2\Delta_\phi - \lv\nabla\phi\rv^2 + 2\lambda(\phi - n) .
\end{align*}
It is readily computed that we again have the Bianchi-type identity $\delta_\phi\cRic_\phi = \frac{1}{2}d\cR_\phi$.

The main point of these definitions is that they provide a concise way of discussing Perelman's $\mW$-functional and gradient Ricci solitons.  We say that $(M^n,g)$ is a \emph{gradient Ricci soliton} if there is a function $\phi\in C^\infty(M)$ and a constant $\lambda\in\bR$ such that the manifold with density $(M^n,g,e^{-\phi}\dvol)$ with parameter $\lambda$ satisfies $\cRic_\phi=0$; we also call $(M^n,g,e^{-\phi}\dvol)$ a gradient Ricci soliton if $\cRic_\phi=0$ for some parameter $\lambda\in\bR$.  It is an immediate consequence of the above Bianchi-type identity that $\cR_\phi$ is constant for gradient Ricci solitons.  Perelman's $\mW$-functional $\mW\colon\Met(M)\times C^\infty(M)\times\bR_+\to\bR$ is defined on a compact manifold $M^n$ by
\[ \mW(g,\phi,\tau) = \int_M \tau\cR_\phi\,(4\pi\tau)^{-\frac{n}{2}}e^{-\phi}\dvol_g , \]
where $\Met(M)$ is the space of Riemannian metrics on $M$ and $\cR_\phi$ is the modified weighted scalar curvature of the manifold with density $(M^n,g,e^{-\phi}\dvol)$ with parameter $\lambda=\frac{1}{2\tau}$.  Critical points of $\mW$ restricted to the class
\begin{equation}
\label{eqn:measure_1}
\mC_1 := \left\{ (g,\phi,\tau) \in \Met(M)\times C^\infty(M)\times\bR_+ \colon \int_M (4\pi\tau)^{-\frac{n}{2}}e^{-\phi}\dvol_g = 1 \right\}
\end{equation}
of volume-normalized pairs of metrics and measures-with-scale are shrinking gradient Ricci solitons.  For $g$ fixed, critical points of $\mW$ restricted to the class
\begin{equation}
\label{eqn:conf_1}
\mC_1(g) = \left\{ (\phi,\tau)\in C^\infty(M)\times\bR_+ \colon (g,\phi,\tau)\in\mC_1 \right\}
\end{equation}
of volume-normalized measures-with-scale have constant modified weighted scalar curvature.  Regarding the $\mW$-functional as the weighted analogue of the total scalar curvature functional provides a strong analogy between variations of the measure $e^{-\phi}\dvol$ and the scale $\tau$ --- equivalently, variations of $(\phi,\tau)$ --- and variations of a conformal metric in connection to the problems of finding critical points for the $\mW$-functional in $\mC_1(g)$ and the Yamabe Problem, respectively.  A more precise description of this analogy was developed in~\cite{Case2013y} through the language of smooth metric measure spaces.

We pursue in a strong way the idea that the proper analogue of the notion of a pointwise conformal change of metric for manifolds with density is a pointwise change of the measure $e^{-\phi}\dvol$.  Thus we define a manifold with density $(M^n,g,e^{-\phi}\dvol)$ to be \emph{locally conformally flat in the weighted sense} if $(M^n,g)$ is flat; i.e.\ if the Riemann curvature tensor vanishes identically.  Phrased differently, the only way that $(M^n,g,e^{-\phi}\dvol)$ is locally equivalent to the flat space $(\bR^n,dx^2,\dvol)$ up to changes of measure is if $(M^n,g)$ is locally isometric to flat Euclidean space.  An alternative motivation for this definition which passes through the language of smooth metric measure spaces can be found in~\cite{Case2011t}.

The article~\cite{Case2011t} provides two additional definitions we require.  Given a manifold with density $(M^n,g,e^{-\phi}\dvol)$ and a parameter $\lambda\in\bR$, the weighted analogue of the Cotton tensor is $d\cRic_\phi$; i.e.\ the tensor defined by
\[ d\cRic_\phi\left(X,Y,Z\right) := \nabla_X\cRic_\phi(Y,Z) - \nabla_Y\cRic_\phi(X,Z) \]
for all $p\in M$ and all $X,Y,Z\in T_pM$.  A more na\"ive way to see that $d\cRic_\phi$ is the weighted analogue of the Cotton tensor is as follows: just like the divergence of the Weyl tensor is (a multiple of) the Cotton tensor, it is readily computed that the weighted divergence of the Riemann curvature tensor $\Rm$ is
\[ \delta_\phi\Rm = d\cRic_\phi . \]
In particular, if $(M^n,g)$ is a flat manifold, then for any $\phi\in C^\infty(M)$ the manifold with density $(M^n,g,e^{-\phi}\dvol)$ is such that $d\cRic_\phi=0$ for all parameters $\lambda\in\bR$.  The \emph{weighted Bach tensor} $\cB_\phi$ is defined by
\begin{equation}
\label{eqn:weighted_bach}
\cB_\phi := \delta_\phi d\cRic_\phi + \Rm\cdot\cRic_\phi,
\end{equation}
where
\begin{align*}
\left(\delta_\phi d\cRic_\phi\right)(X,Y) & := \sum_{j=1}^n \nabla_{E_i}d\cRic_\phi(E_i,X,Y) - d\cRic_\phi(\nabla\phi,X,Y) , \\
\left(\Rm\cdot\cRic_\phi\right)(X,Y) & := \sum_{j=1}^n \Rm(E_i,X,E_j,Y)\cRic_\phi(E_i,E_j),
\end{align*}
for all $p\in M$ and all $X,Y\in T_pM$, where $\{E_i\}\subset T_pM$ is an orthonormal basis and our sign convention is such that $\Rm\cdot g = \Ric$.  The simple explanation for this definition is that the weighted Bach tensor~\eqref{eqn:weighted_bach} is defined in terms of the weighted Cotton tensor $d\cRic_\phi$ and the weighted Weyl tensor $\Rm$ in the analogous way that the Bach tensor is usually defined.  A more meaningful explanation for this definition is that the weighted Bach tensor appears in the first variation of the total weighted $\sigma_2$-curvature functional in exactly the analogous way as in the Riemannian case; see Proposition~\ref{prop:first_variation2} for details.

The positively-curved flat models of manifolds with density are the \emph{shrinking Gaussians}
\[ \left( \bR^n, dx^2, \exp\left(-\frac{\lv x-x_0\rv^2}{4\tau}\right)\dvol \right) \]
for any constant $\tau>0$ and any fixed point $x_0\in\bR^n$.  It is readily computed that with the parameter $\lambda=\frac{1}{2\tau}$ both the modified Bakry-\'Emery Ricci tensor and the modified weighted scalar curvature vanish identically, and moreover, for all $\tau>0$ we have that
\[ \int_{\bR^n} (4\pi\tau)^{-\frac{n}{2}}e^{-\frac{\lv x-x_0\rv^2}{4\tau}}\dvol = 1 . \]
These normalizations play an important role in motivating the definitions of the weighted $\sigma_k$-curvatures and the associated total weighted $\sigma_k$-curvature functionals given in Section~\ref{sec:defn}.
\section{The Weighted $\sigma_k$-curvature}
\label{sec:defn}

We now define the weighted $\sigma_k$-curvatures and associated objects.  In particular, we define the $k$-th weighted Newton tensors and natural trace-adjustments thereof which lead to conservation laws for the weighted $\sigma_k$-curvatures (see Section~\ref{sec:first_variation} and Section~\ref{sec:divergence}).  We also define the total weighted $\sigma_k$-curvature functionals, which should be regarded as the fully nonlinear analogues of Perelman's $\mW$-functional (see Subsection~\ref{subsec:defn/mW} and Section~\ref{sec:first_variation}).

There are two intuitions behind our definition of the weighted $\sigma_k$-curvatures.  The first is that they are in some sense the infinite-dimensional limits of the usual $\sigma_k$-curvatures, as can be obtained by considering warped products with spaceforms of dimension tending to infinity; this is the intuition partially suggested in Section~\ref{sec:polynomial} and developed more fully in~\cite{Case2014s}.  The second is that the weighted $\sigma_k$-curvature functionals should be defined in terms of the $k$-th weighted elementary symmetric polynomials in such a way that they can be regarded as perturbations of the $k$-th elementary symmetric polynomial of the Bakry-\'Emery Ricci tensor of a manifold with density $(M^n,g,e^{-\phi}\dvol)$.  To that end, observe that for any parameter $\lambda$, the modified weighted scalar curvature is
\[ \frac{1}{2}\cR_\phi = \tr\cRic_\phi + \cY_\phi \]
for
\begin{equation}
\label{eqn:cY}
\cY_\phi := -\frac{1}{2}\left(R + \lv\nabla\phi\rv^2 - 2\lambda\phi\right) .
\end{equation}
In particular, we see that $\cY_\phi$ depends only on derivatives up to order one of $\phi$, so it becomes natural to consider the weighted $\sigma_1$-curvature to be $\frac{1}{2}\cR_\phi$.  This leads to the following definition.

\begin{defn}
\label{defn:sigmak}
Given $k\in\bN_0$, the \emph{weighted $\sigma_k$-curvature} $\csigma_{k,\phi}$ of a manifold with density $(M^n,g,e^{-\phi}\dvol)$ with parameter $\lambda$ is
\[ \csigma_{k,\phi} := \sigma_k^\infty\left(\cY_\phi;\cRic_\phi\right) \]
for $\sigma_k^\infty$ the $k$-th weighted elementary symmetric polynomial, $\cY_\phi$ the scalar function~\eqref{eqn:cY}, and $\cRic_\phi$ the modified Bakry-\'Emery Ricci tensor.
\end{defn}

We denote by $\sigma_{k,\phi}$ the weighted $\sigma_k$-curvature of $(M^n,g,e^{-\phi}\dvol)$ with parameter $\lambda=0$.  We have opted not to use the terminology ``modified weighted $\sigma_k$-curvature'' because we find it too cumbersome.  Instead, we mark the difference between ``unmodified'' and ``modified'' weighted $\sigma_k$-curvatures by omitting or including, respectively, the tilde in our notation.  Similar conventions are used for the weighted $k$-th Newton tensor $\cT_{k,\phi}$ and its trace-adjusted version $\cE_{k,\phi}$ below.

The weighted $k$-th Newton tensor is the weighted $k$-th Newton transformation of $\cRic_\phi$ and $\cY_\phi$.  Writing it only in terms of the modified Bakry-\'Emery Ricci tensor and the weighted $\sigma_k$-curvatures produces the following definition.

\begin{defn}
\label{defn:Tk}
Given $k\in\bN_0$, the \emph{weighted $k$-th Newton tensor} $\cT_{k,\phi}$ of a manifold with density $(M^n,g,e^{-\phi}\dvol)$ with parameter $\lambda$ is
\[ \cT_{k,\phi} = \sum_{j=0}^k (-1)^j\csigma_{k-j,\phi}\cRic_\phi^j \]
where $\cRic_\phi^j$ denotes the $j$-fold composition of $\cRic_\phi$, regarded as an endomorphism of $TM$, with itself.
\end{defn}

In conformal geometry, one is often more interested in the tracefree part of the $k$-th Newton tensors (cf.\ \cite{ChangGurskyYang2003b}).  The analogous tensor in the weighted setting is obtained by subtracting the metric term from $\cT_{k,\phi}$.

\begin{defn}
\label{defn:Ek}
Let $(M^n,g,e^{-\phi}\dvol)$ be a manifold with density, let $\lambda\in\bR$, and let $k\in\bN$.  The \emph{trace-adjusted weighted $k$-Newton tensor $\cE_{k,\phi}$} is defined by
\begin{equation}
\label{eqn:Ek}
\cE_{k,\phi} := \cT_{k,\phi} - \csigma_{k,\phi}g .
\end{equation}
\end{defn}

As in Riemannian geometry, the PDE prescribing the weighted $\sigma_k$-curvature is not always elliptic (cf.\ \cite{Viaclovsky2000}).  There are natural cones in which this PDE is elliptic.

\begin{defn}
\label{defn:weighted_elliptic_cone}
Let $(M^n,g)$ be a Riemannian manifold and let $k\in\bN$.  The \emph{negative weighted elliptic $k$-cone $\Gamma_k^{\infty,-}$} is the set
\[ \Gamma_k^{\infty,-} := \left\{ (\phi,\lambda) \in C^\infty(M)\times\bR \colon \left(\cY(p),\cRic_\phi\rv_p\right)\in\Gamma_k^{\infty,-} \text{ for all $p\in M$} \right\} , \]
where we define $\cY$ and $\cRic_\phi$ in terms of the manifold with density $(M^n,g,e^{-\phi}\dvol)$ and the parameter $\lambda$ and we interpret the statement $\left(\cY(p),\cRic\rv_p\right)\in\Gamma_k^{\infty,-}$ using Definition~\ref{defn:endomorphism_elliptic_cone}.

The manifold with density $(M^n,g,e^{-\phi}\dvol)$ with parameter $\lambda\in\bR$ \emph{lies in the negative weighted elliptic $k$-cone} if $(\phi,\lambda)\in\Gamma_k^{\infty,-}$.
\end{defn}

Note that, from the definition of the negative weighted elliptic cones given in Section~\ref{sec:polynomial}, we have the sequence of inclusions
\[ \Gamma_1^{\infty,-} \supset \Gamma_2^{\infty,-} \supset \Gamma_3^{\infty,-} \supset \dotsb . \]

We have chosen to focus on the \emph{negative} weighted elliptic $k$-cones, rather than the \emph{positive} weighted elliptic $k$-cones, because with our definitions shrinking gradient Ricci solitons are always in $\Gamma_k^{\infty,-}$ for all $k\in\bN$.  More precisely, Perelman~\cite{Perelman1} showed that the $\nu$-entropy of any compact manifold is always strictly negative, and hence, by Theorem~\ref{thm:obata/1}, $\csigma_{1,\phi}<0$ for any compact gradient Ricci soliton; it is trivial to conclude from this that compact gradient Ricci solitons lie in $\Gamma_k^{\infty,-}$ for all $k\in\bN$ (cf.\ Lemma~\ref{lem:modified_sigmak_grs}).  For a given compact gradient Ricci soliton, one could modify our definition of $\cY_\phi$ by adding a constant $c$ so that $\csigma_{1,\phi}^{(c)}>0$ for $\csigma_{1,\phi}^{(c)}$ as in Remark~\ref{rk:divergence}, and hence study gradient Ricci solitons as elements of the positive weighted elliptic $k$-cones.  However, there does not seem to be a uniform choice of $c$ which works for \emph{every} compact gradient Ricci soliton, so this approach does not seem natural.  Nevertheless, in solving the weighted $\sigma_k$-curvature prescription problem one also needs to impose some sort of positivity as, for example, minimizers of Perelman's $\mW$-functional do not exist when $\nu(g)=-\infty$, though certainly there exist many elements of the negative weighted elliptic $1$-cone.

As stated above, the point of restricting to the negative elliptic cones is that the class $\Gamma_k^{\infty,-}$ gives a natural set of functions and parameters for which the weighted $\sigma_k$-curvature prescription problem is elliptic.

\begin{prop}
\label{prop:elliptic_cone}
Let $(M^n,g)$ be a Riemannian manifold and fix $k\in\bN$ and $\lambda\in\bR$.  The operator
\begin{equation}
\label{eqn:sigmak_pde}
\phi \mapsto \csigma_{k,\phi}
\end{equation}
is elliptic in the cone $\left\{\phi\colon (\phi,\lambda)\in\Gamma_k^{\infty,-}\right\}$.
\end{prop}

\begin{proof}

As an immediate consequence of Proposition~\ref{prop:diff}, we see that the linearization of the operator~\eqref{eqn:sigmak_pde} is $\cT_{k-1,\phi}$.  Proposition~\ref{prop:ellipticity} implies that $\cT_{k-1,\phi}$ is definite; i.e.\ \eqref{eqn:sigmak_pde} is elliptic.
\end{proof}

\subsection{The total weighted $\sigma_k$-curvature functionals}
\label{subsec:defn/mW}

In the cases when the weighted $\sigma_k$-curvature is variational (see Theorem~\ref{thm:variational_status}), it is natural to expect that critical points of the total weighted $\sigma_k$-curvature functionals for a fixed metric have constant weighted $\sigma_k$-curvature (cf.\ \cite{BransonGover2008}).  It turns out that this is not the case, but it is true up to lower order terms; see the discussion at the end of Section~\ref{sec:variational_status} for details.  Nevertheless, we expect the total weighted $\sigma_k$-curvature functionals to be of interest, and so we shall endow them with a special notation.

\begin{defn}
\label{defn:mW}
Let $M^n$ be a compact manifold and let $k\in\bN$.  The $\mW_k$-functional $\mW_k\colon\Met(M)\times C^\infty(M)\times\bR_+\to\bR$ is defined by
\begin{equation}
\label{eqn:mW}
\mW_k(g,\phi,\tau) = \int_M \tau^k\csigma_{k,\phi} (4\pi\tau)^{-\frac{n}{2}} e^{-\phi}\dvol_g ,
\end{equation}
where $\csigma_{k,\phi}$ is the weighted $\sigma_k$-curvature of $(M^n,g,e^{-\phi}\dvol)$ with parameter $\lambda=\frac{1}{2\tau}$.
\end{defn}

We are primarily interested in the special cases $k\in\{1,2\}$, as it is only in these cases that the weighted $\sigma_k$-curvature is variational on all Riemannian manifolds; see Theorem~\ref{thm:variational_status}.  For these reasons, we say some more words about these cases in particular.

It is readily computed that the $\mW_1$-functional is equivalently written
\[ \mW_1(g,\phi,\tau) = \frac{1}{2}\int_M \left[ \tau\left( R + \lv\nabla\phi\rv^2\right) + \phi - n \right] (4\pi\tau)^{-\frac{n}{2}}\dvol , \]
so that $\mW_1$ is one half of Perelman's $\mW$-functional.  It is also readily computed that the $\mW_2$-functional is equivalently written
\begin{align*}
\mW_2(g,\phi,\tau) & = \int_M \bigg[ \tau^2\sigma_{2,\phi} + \frac{\tau}{2}\left(\phi-n+2\right)\sigma_{1,\phi} \\
& \qquad\qquad + \frac{1}{8}\left(\phi^2-2n\phi+n(n-1)\right)\bigg](4\pi\tau)^{-\frac{n}{2}}e^{-\phi}\dvol ,
\end{align*}
with further equivalent formulations available by using the definition of the (unmodified) weighted $\sigma_1$- and $\sigma_2$-curvatures and integration by parts.  A key point of this definition is that shrinking gradient Ricci solitons are critical points for the $\mW_2$-functional restricted to $\mC_1(g)$; see Proposition~\ref{prop:first_variation2}.

The reason for the particular power of $\tau$ in the definition of the $\mW_k$-functional is to ensure that $\mW_k$ satisfies the same scale invariance as Perelman's $\mW$-functional.  More precisely, a simple computation, which we omit, verifies the following lemma.

\begin{lem}
\label{lem:mW_scale_invariance}
Let $(M^n,g)$ be a compact Riemannian manifold.  It holds for all $\phi\in C^\infty(M)$ and all $c,\tau>0$ that
\begin{equation}
\label{eqn:mW_scale_invariance}
\mW_k(cg,\phi,\tau) = \mW_k(g,\phi,c^{-1}\tau) .
\end{equation}
\end{lem}

For the purposes of this article, the main consequence of Lemma~\ref{lem:mW_scale_invariance} is that in order to compute the first variation of the $\mW_k$-functional, it suffices to compute the first variation through variations of the metric $g$ and the measure $\phi$ only.

\subsection{Gradient Ricci solitons}
\label{subsec:defn/grs}

We conclude this section with some useful facts about gradient Ricci solitons.  Recall that $(M^n,g,e^{-\phi}\dvol)$ is a gradient Ricci soliton if $\Ric_\phi=\lambda g$ for some constant $\lambda\in\bR$.  We are particularly interested in studying volume-normalized shrinking gradient Ricci solitons.

\begin{defn}
A Riemannian manifold with density $(M^n,g,e^{-\phi}\dvol)$ is a \emph{volume-normalized shrinking gradient Ricci soliton} if there is a constant $\tau>0$ such that $\Ric_\phi=\frac{1}{2\tau}g$ and
\[ \int_M (4\pi\tau)^{-\frac{n}{2}}e^{-\phi}\dvol = 1 . \]
\end{defn}

In particular, the Gaussian shrinkers are volume-normalized shrinking gradient Ricci solitons.  Note that if $(M^n,g,e^{-\phi}\dvol)$ satisfies $\Ric_\phi=\frac{1}{2\tau}g>0$, then the weighted volume $\int_M e^{-\phi}\dvol$ is finite~\cite{Morgan2005}, and hence we can add a constant to $\phi$ to ensure that $(M^n,g,e^{-\phi}\dvol)$ is volume-normalized.

We establish here three facts.  First, gradient Ricci solitons have constant weighted $\sigma_k$-curvature when the parameter is suitably chosen, and moreover, all volume-normalized shrinking gradient Ricci solitons except the shrinking Gaussians lie in all the negative weighted elliptic $k$-cones.  That the shrinking Gaussians only lie on the boundary of the negative weighted elliptic $k$-cones is not a problem in our definitions, as is made clear in Section~\ref{sec:obata} and Section~\ref{sec:local_extrema}.

\begin{lem}
\label{lem:modified_sigmak_grs}
Let $(M^n,g,e^{-\phi}\dvol)$ be a gradient Ricci soliton with $\Ric_\phi=\lambda g$.  For each $k\in\bN$, let $\csigma_{k,\phi}$ be the modified weighted $\sigma_k$-curvature determined by $\lambda$.  Then
\begin{equation}
\label{eqn:modified_sigmak_grs}
\csigma_{k,\phi} = \frac{1}{k!}\left(\csigma_{1,\phi}\right)^k
\end{equation}
is constant.  Moreover, if $(M^n,g,e^{-\phi}\dvol)$ is a volume-normalized shrinking gradient Ricci soliton, then either $(g,\phi)\in\Gamma_k^{\infty,-}$ for all $k\in\bN$ or $(M^n,g,e^{-\phi}\dvol)$ is isometric to a shrinking Gaussian, in which case $\csigma_{k,\phi}=0$ for all $k\in\bN$.
\end{lem}

\begin{proof}

By assumption $\cRic_\phi=0$, and hence $\csigma_{k,\phi}=\sigma_k^\infty(\cY_\phi;0)$.  The conclusion~\eqref{eqn:modified_sigmak_grs} follows immediately from Proposition~\ref{prop:ordinary_sigmak}, while the fact $\delta_\phi\cRic_\phi=d\csigma_{1,\phi}$ implies that~\eqref{eqn:modified_sigmak_grs} is constant.  Lastly, Yokota~\cite{Yokota2008,Yokota2008add} showed that for any volume-normalized shrinking gradient Ricci soliton $(M^n,g,e^{-\phi}\dvol)$, it holds that $\csigma_{1,\phi}\leq0$ with equality if and only if $(M^n,g,e^{-\phi}\dvol)$ is isometric to a shrinking Gaussian.  Inserting this into~\eqref{eqn:modified_sigmak_grs} yields the final claim.
\end{proof}

Second, the potential $\phi$ of a volume-normalized shrinking gradient Ricci soliton is an $L^2$-eigenvalue of the weighted Laplacian.

\begin{lem}
\label{lem:grs_potential}
Let $(M^n,g,e^{-\phi}\dvol)$ be a volume-normalized shrinking gradient Ricci soliton such that $\Ric_\phi=\frac{1}{2\tau}g$.  Let $\phi_0:=\phi-\frac{n}{2}-2\tau\csigma_{1,\phi}$ be the mean-free (with respect to $e^{-\phi}\dvol$) part of $\phi$.  Then
\begin{align}
\label{eqn:grs_potential_eigenvalue} -\Delta_\phi\phi_0 & = \frac{1}{\tau}\phi_0, \\
\label{eqn:grs_potential_inequality} \int_M \phi_0^2 (4\pi\tau)^{-\frac{n}{2}}e^{-\phi}\dvol & \leq \frac{n}{2} .
\end{align}
Moreover, equality holds in~\eqref{eqn:grs_potential_inequality} if and only if $(M^n,g,e^{-\phi}\dvol)$ is a shrinking Gaussian.
\end{lem}

\begin{proof}

By definition,
\[ 2\tau\csigma_{1,\phi} = \tau\left(\Delta_\phi\phi + \tr\Ric_\phi\right) + \phi - n, \]
from which~\eqref{eqn:grs_potential_eigenvalue} immediately follows.  From~\cite{PigolaRimoldiSetti2011} we know that the scalar curvature of $(M^n,g,e^{-\phi}\dvol)$ is nonnegative and vanishes if and only if $(M^n,g,e^{-\phi}\dvol)$ is a shrinking Gaussian, and hence
\begin{equation}
\label{eqn:grs_potential_pw_inequality}
2\tau\csigma_{1,\phi} = -\tau\left(R + \lv\nabla\phi\rv^2\right) + \phi \leq -\tau\lv\nabla\phi\rv^2 + \phi
\end{equation}
with the same characterization of equality.  From~\cite{CaoZhou2010} we know that $\phi,\lv\nabla\phi\rv\in L^2(M,e^{-\phi}\dvol)$.  Hence~\eqref{eqn:grs_potential_eigenvalue} and~\eqref{eqn:grs_potential_pw_inequality} imply that
\[ \int_M \phi_0^2(4\pi\tau)^{-\frac{n}{2}}e^{-\phi}\dvol = \tau\int_M \lv\nabla\phi_0\rv^2(4\pi\tau)^{-\frac{n}{2}} e^{-\phi}\dvol \leq \frac{n}{2} \]
with equality if and only if $(M^n,g,e^{-\phi}\dvol)$ is a shrinking Gaussian.
\end{proof}

Third, the total unmodified weighted $\sigma_1$- and $\sigma_2$-curvatures are readily computed for volume-normalized shrinking gradient Ricci solitons.

\begin{lem}
\label{lem:total_sigmak_grs}
Let $(M^n,g,e^{-\phi}\dvol)$ be a volume-normalized shrinking gradient Ricci soliton with $\Ric_\phi=\frac{1}{2\tau}g$.  Then
\begin{align}
\label{eqn:total_sigma1_grs} \int_M \sigma_{1,\phi}(4\pi\tau)^{-\frac{n}{2}}e^{-\phi}\dvol & = \frac{n}{4\tau} , \\
\label{eqn:total_sigma2_grs} \int_M \sigma_{2,\phi}(4\pi\tau)^{-\frac{n}{2}}e^{-\phi}\dvol & = \frac{n(n-4)}{32\tau^2} + \frac{1}{8\tau^2}\int_M \phi_0^2(4\pi\tau)^{-\frac{n}{2}}e^{-\phi}\dvol
\end{align}
for $\phi_0$ as in Lemma~\ref{lem:grs_potential}.
\end{lem}

\begin{proof}

To begin, observe that
\begin{equation}
\label{eqn:unmodified_sigma1}
\sigma_{1,\phi} = -\frac{1}{2\tau}\left(\phi_0-\frac{n}{2}\right) .
\end{equation}
Inserting this into the definition of $\sigma_{2,\phi}$ immediately yields
\begin{equation}
\label{eqn:unmodified_sigma2}
\sigma_{2,\phi} = \frac{1}{8\tau^2}\left( \phi_0^2 - n\phi_0 + \frac{n(n-4)}{4}\right) .
\end{equation}
The result now follows by integration.
\end{proof}

\begin{remark}

An intriguing corollary of Lemma~\ref{lem:total_sigmak_grs} is that the total unmodified weighted $\sigma_2$-curvature characterizes nontrivial compact gradient Ricci solitons in dimension four: if $(M^4,g,e^{-\phi}\dvol)$ is a compact gradient Ricci soliton, then $\int\sigma_{2,\phi}>0$ if and only if $\phi$ is nonconstant.
\end{remark}
\section{Variations of the total weighted $\sigma_1$- and $\sigma_2$-curvatures}
\label{sec:first_variation}

The variational structure of the weighted $\sigma_k$-curvatures is important for many reasons.  For example, it gives rise to formulae relating the derivative of weighted $\sigma_k$-curvatures to weighted divergences of certain sections of $S^2T^\ast M$ (cf.\ \cite[Paragraph~4.10]{Besse}) and it suggests a natural route to solving curvature prescription problems involving the weighted $\sigma_k$-curvatures (cf. \cite{GeWang2006,GuanWang2004,ShengTrudingerWang2007}).  The first application is of particular relevance to this article, as these divergence formulae play an important role in proving Obata-type theorems (cf.\ Section~\ref{sec:obata}) and also in proving that the weighted $\sigma_k$-curvatures are variational (cf.\ Section~\ref{sec:variational_status}).

In this section we compute explicitly the first variations of the total weighted $\sigma_k$-curvature functionals for $k\in\{1,2\}$.  When $k=1$ this is due to Perelman~\cite{Perelman1}, and the formula has many applications in the study of the Ricci flow beyond a simple divergence formula.  While we only use our computation in the case $k=2$ for the divergence formula and to characterize the critical points of $\mW_2\colon\mC_1(g)\to\bR$, we expect this formula to find further applications.  The divergence formulae derived in this section are used to motivate the general divergence formulae for the weighted Newton tensors on manifolds with density derived in Section~\ref{sec:divergence}.  These in turn provide one of the main ingredients in the characterization given in Section~\ref{sec:variational_status} of when the weighted $\sigma_k$-curvature is variational.

\subsection{The first variation of the total weighted $\sigma_1$-curvature}
\label{subsec:first_variation/1}

As the total weighted $\sigma_1$-curvature functional is Perelman's $\mW$-functional, the results of this subsection can all be found in~\cite{Perelman1}.  Nevertheless, it is convenient to recall the first variation of $\mW_1$ and its proof in order to motivate the computation of the first variation of $\mW_2$ given in Subsection~\ref{subsec:first_variation/2} below.

We say that $(g_t,\phi_t)$ with $t\in(-\varepsilon,\varepsilon)$ is a \emph{variation of the manifold with density $(M^n,g,e^{-\phi}\dvol)$} if $g_t$ is a smooth family of Riemannian metrics on $M$, $\phi_t$ is a smooth family of smooth functions on $M$, and $g_0=g$ and $\phi_0=\phi$.  We denote by ${}^\bullet$ the operator $\frac{d}{dt}\rv_{t=0}$; e.g.\ $\cR_\phi^\bullet$ denotes
\[ \cR_\phi^\bullet := \left.\frac{d}{dt}\right|_{t=0} \bigl(\cR_\phi\bigr)_t \]
for $(\cR_\phi)_t$ the weighted scalar curvature of $(M^n,g_t,e^{-\phi_t}\dvol_{g_t})$ with fixed parameter $\lambda\in\bR$.

The main ingredient in the computation of the first variation of the total weighted $\sigma_1$-curvature is the following formula for the linearization of the weighted scalar curvature.

\begin{lem}
\label{lem:simple_variation1}
Let $(M^n,g,e^{-\phi}\dvol)$ be a manifold with density and let $\lambda\in\bR$.  Given a variation $(g_t,\phi_t)$ with $g^\bullet=h$ and $\phi^\bullet=\psi$, it holds that
\begin{equation}
\label{eqn:simple_variation1}
\cR_\phi^\bullet = -\left\lp\cRic_\phi,h\right\rp + \delta_\phi^2h + 2(\Delta_\phi+\lambda)\left(\psi-\frac{1}{2}\tr h\right) .
\end{equation}
\end{lem}

\begin{proof}

This follows immediately from the formulae for $R^\bullet$ and $(\Delta\phi)^\bullet$ as can be found, for example, in~\cite[Section~1.K]{Besse}.
\end{proof}

Using also the well-known formula for the first variation of the Riemannian volume element leads to the formula for the first variation of the total weighted $\sigma_1$-curvature functional.

\begin{prop}
\label{prop:first_variation1}
Let $(M^n,g,e^{-\phi}\dvol)$ be a smooth metric measure space and let $\lambda\in\bR$.  Given a variation $(g_t,\phi_t)$ with $g^\bullet=h$ and $\phi^\bullet=\psi$, it holds that
\[ \mW_1^\bullet = \int_M \bigg[ \frac{\tau}{2}\left\lp\cE_{1,\phi},h\right\rp - \left(\tau\csigma_{1,\phi} - \frac{1}{2}\right)\left(\psi-\frac{1}{2}\tr h\right) \bigg] (4\pi\tau)^{-\frac{n}{2}}e^{-\phi}\dvol . \]
\end{prop}

\begin{proof}

By definition, $\csigma_{1,\phi}=\frac{1}{2}\cR_\phi$.  Thus
\[ \mW_1^\bullet =\int_M \left[\frac{\tau}{2}\cR_\phi^\bullet - \tau\csigma_{1,\phi}\left(\psi-\frac{1}{2}\tr h\right)\right](4\pi\tau)^{-\frac{n}{2}}e^{-\phi}\dvol . \]
The desired conclusion follows immediately from this, Lemma~\ref{lem:simple_variation1}, and the definition of $\cE_{1,\phi}$.
\end{proof}

Since the total weighted $\sigma_1$-curvature functional is invariant under pullback by diffeomorphisms, we can derive as in~\cite[Paragraph~4.10]{Besse} the following divergence formula for the weighted $\sigma_1$-curvature.  This provides an alternative derivation of the Bianchi identity $\delta_\phi\cRic_\phi=\frac{1}{2}d\cR_\phi$.

\begin{cor}
\label{cor:divE1}
Let $(M^n,g,e^{-\phi}\dvol)$ be a smooth metric measure space and let $\lambda\in\bR$.  Then
\begin{equation}
\label{eqn:divE1}
\delta_\phi\cE_{1,\phi} = -d\csigma_{1,\phi} .
\end{equation}
\end{cor}

\begin{proof}

Let $X$ be a compactly-supported vector field on $M$ and let $\xi_t$ be the one-parameter family of diffeomorphisms generated by $X$.  Let $g_t=\xi_t^\ast g$ and $\phi_t=\xi_t^\ast\phi$, so that $g_t^\bullet=L_Xg$ and $\phi_t^\bullet=X\phi$.  Since $\mW_1$ is invariant under pullback by diffeomorphisms we have that $\mW_1^\bullet=0$ and hence, by Proposition~\ref{prop:first_variation1},
\[ 0 = \int_M \bigg[ \frac{\tau}{2}\left\lp\cE_{1,\phi},L_Xg\right\rp + \left(\tau\csigma_{1,\phi} - \frac{1}{2}\right)\delta_\phi X \bigg] (4\pi\tau)^{-\frac{n}{2}}e^{-\phi}\dvol . \]
Since $X$ is arbitrary, integration by parts yields the desired result.
\end{proof}

The correct analogue of the Yamabe Problem in the weighted setting is to construct critical points of the total weighted $\sigma_1$-curvature functional within the class $\mC_1(g)$ of unit-volume measures-with-scale on $(M^n,g)$ defined in~\eqref{eqn:conf_1}.  Proposition~\ref{prop:first_variation1} allows us to characterize these critical points.

\begin{cor}
\label{cor:critW1}
Let $(M^n,g)$ be a compact Riemannian manifold.  Then $(\phi,\tau)\in\mC_1(g)$ is a critical point of $\mW_1\colon\mC_1(g)\to\bR$ if and only if there is a constant $c\in\bR$ such that
\begin{align}
\label{eqn:critW1_el} \csigma_{1,\phi} - \frac{1}{2\tau} & = c , \\
\label{eqn:critW1_tr} \int_M \tr\cRic_\phi\,e^{-\phi}\dvol & = 0 ,
\end{align}
where $\csigma_{1,\phi}$ and $\cRic_\phi$ are given in terms of the manifold with density $(M^n,g,e^{-\phi}\dvol)$ and the parameter $\lambda=\frac{1}{2\tau}$.
\end{cor}

\begin{proof}

By Lemma~\ref{lem:mW_scale_invariance}, variations of $(\phi,\tau)$ within $\mC_1(g)$ are equivalent to variations of $(g,\phi)$ such that $h:=g^\bullet$ satisfies $h=cg$ and the constraint
\begin{equation}
\label{eqn:constraint}
\int_M \left(\psi - \frac{1}{2}\tr h\right)(4\pi\tau)^{-\frac{n}{2}}e^{-\phi}\dvol = 0
\end{equation}
holds.  The conclusion then follows immediately from Proposition~\ref{prop:first_variation1}.
\end{proof}

\subsection{The first variation of the total weighted $\sigma_2$-curvature}
\label{subsec:first_variation/2}

The computation of the first variation of $\mW_2$ relies on two lemmas.  The first lemma computes the linearization of the Bakry-\'Emery Ricci tensor (cf.\ \cite{CaoZhu2010}).

\begin{lem}
\label{lem:simple_variation2}
Let $(M^n,g,e^{-\phi}\dvol)$ be a smooth metric measure space and let $\lambda\in\bR$.  Given a variation $(g_t,\phi_t)$ with $g^\bullet=h$ and $\phi^\bullet=\psi$, it holds that
\begin{equation}
\label{eqn:simple_variation2}
\cRic_\phi^\bullet = -\frac{1}{2}\Delta_\phi h - \frac{1}{2}\cRic_\phi \hash h - \Rm\cdot h + \frac{1}{2}L_{\delta_\phi h}g + \nabla^2\left(\psi-\frac{1}{2}\tr h\right)
\end{equation}
for $\hash$ the natural action of $\End(TM)$ on $TM$ extended to $S^2T^\ast M$ as a derivation; i.e.
\[ \left(\cRic_\phi\hash h\right)(X,Y) := -\sum_{i=1}^n \left(\cRic_\phi(E_i,X)h(E_i,Y) + \cRic_\phi(E_i,Y)h(X,E_i)\right) \]
for all $p\in M$ and all $X,Y\in T_pM$ with $\{E_1,\dotsc,E_n\}$ an orthonormal basis of $T_pM$.
\end{lem}

\begin{proof}

This follows immediately from the formulae for $\Ric^\bullet$ and $\nabla^\bullet$ as can be found in~\cite[Section~1.K]{Besse}.
\end{proof}

The second lemma, which is used primarily to write the gradient of $\mW_2$ in terms of the weighted Bach tensor, is the following weighted Weitzenb\"ock formula for sections of $S^2T^\ast M$.

\begin{lem}
\label{lem:weitzenbock}
Let $(M^n,g,e^{-\phi}\dvol)$ be a manifold with density and let $T$ be a section of $S^2T^\ast M$.  Then
\[ \Delta_\phi T = \delta_\phi dT + \frac{1}{2}L_{\delta_\phi T}g - \frac{1}{2}\Ric_\phi\hash T - \Rm\cdot T . \]
\end{lem}

\begin{proof}

Symmetrizing~\cite[Proposition~4.1]{Bourguignon1981} shows that
\[ \Delta T = \delta dT + \frac{1}{2}L_{\delta T}g - \frac{1}{2}\Ric\hash T - \Rm\cdot T . \]
On the other hand, it is readily computed that
\[ \delta_\phi dT + \frac{1}{2}L_{\delta_\phi T}g = \delta dT + \frac{1}{2}L_{\delta T}g - \nabla_{\nabla\phi}T + \frac{1}{2}\nabla^2\phi\hash T . \]
Combining these two displays yields the desired result.
\end{proof}

We are ready to derive the first variation of the total weighted $\sigma_2$-curvature.

\begin{prop}
\label{prop:first_variation2}
Let $(M^n,g,e^{-\phi}\dvol)$ be a smooth metric measure space and let $\lambda\in\bR$.  Given a variation $(g_t,\phi_t)$ with $g^\bullet=h$ and $\phi^\bullet=\psi$, it holds that
\begin{equation}
\label{eqn:first_variation2}
\begin{split}
\mW_2^\bullet & = \int_M \bigg[\frac{\tau^2}{2}\left\lp \cB_\phi + \cE_{2,\phi} - \frac{1}{2\tau}\cE_{1,\phi}, h\right\rp \\
& \quad - \left(\tau^2\csigma_{2,\phi} - \frac{\tau}{2}\csigma_{1,\phi}\right)\left(\psi-\frac{1}{2}\tr h\right)\bigg](4\pi\tau)^{-\frac{n}{2}}e^{-\phi}\dvol ,
\end{split}
\end{equation}
where $\cB_\phi$ is the weighted Bach tensor~\eqref{eqn:weighted_bach} and $\cE_{k,\phi}$ are the trace-adjusted weighted Newton tensors~\eqref{eqn:Ek}.
\end{prop}

\begin{proof}

By definition,
\[ \csigma_{2,\phi} = \frac{1}{8}\cR_\phi^2 - \frac{1}{2}\left|\cRic_\phi\right|^2 . \]
Applying Lemma~\ref{lem:simple_variation1} and Lemma~\ref{lem:simple_variation2} yields
\begin{align*}
\csigma_{2,\phi}^\bullet & = -\frac{1}{4}\left\lp\cR_\phi\cRic_\phi,h\right\rp + \frac{1}{4}\cR_\phi\delta_\phi^2 h + \frac{1}{2}\cR_\phi\left(\Delta_\phi + \frac{1}{2\tau}\right)\left(\psi-\frac{1}{2}\tr h\right) \\
& \quad + \frac{1}{2}\left\lp\cRic_\phi,\Delta_\phi h + 2\Rm\cdot h - L_{\delta_\phi h}g - 2\nabla^2\left(\psi-\frac{1}{2}\tr h\right)\right\rp .
\end{align*}
Multiplying by $\tau^2$ and integrating with respect to $(4\pi\tau)^{-\frac{n}{2}}e^{-\phi}\dvol$ then yields
\begin{align*}
\mW_2^\bullet & = \int_M \bigg[\tau^2\left\lp \frac{1}{2}\Delta_\phi\cRic_\phi - \frac{1}{4}\nabla^2\cR_\phi + \Rm\cdot\cRic_\phi - \frac{1}{4}\cR_\phi\cRic_\phi,h \right\rp \\
& \qquad - \left(\tau^2\csigma_{2,\phi} - \frac{\tau}{2}\csigma_{1,\phi}\right)\left(\psi-\frac{1}{2}\tr h\right) \bigg] (4\pi\tau)^{-\frac{n}{2}}e^{-\phi}\dvol .
\end{align*}
The final conclusion~\eqref{eqn:first_variation2} is an immediate consequence of Lemma~\ref{lem:weitzenbock}.
\end{proof}

By arguing as in the proof of Corollary~\ref{cor:divE1} we arrive at the following divergence formula.  Note that this identity can also be deduced from the Ricci identity and the Bianchi identity $\delta_\phi\cRic_\phi=\frac{1}{2}d\cR_\phi$.

\begin{cor}
\label{cor:divE2}
Let $(M^n,g,e^{-\phi}\dvol)$ be a smooth metric measure space and let $\lambda\in\bR$.  Then
\begin{equation}
\label{eqn:divE2}
\delta_\phi\left(\cB_\phi + \cE_{2,\phi} - \lambda\cE_{1,\phi}\right) = -d\left(\csigma_{2,\phi} - \lambda\csigma_{1,\phi}\right) .
\end{equation}
\end{cor}

Note that Corollary~\ref{cor:divE1} and Corollary~\ref{cor:divE2} can be combined to identify $d\csigma_{2,\phi}$ as a weighted divergence.  This motivates the general divergence formulae in Section~\ref{sec:divergence}.

We can also use Proposition~\ref{prop:first_variation2} to prove Theorem~\ref{thm:intro/vary_w2}.

\begin{proof}[Proof of Theorem~\ref{thm:intro/vary_w2}]

Suppose first that $(g,\phi,\tau)$ is a critical point of $\mW_2\colon\mC_1\to\bR$.  Considering variations of the metric only implies, via Proposition~\ref{prop:first_variation2}, that~\eqref{eqn:intro/vary_w2/einstein} holds.  Conversely, if $(g,\phi,\tau)\in\mC_1$ is such that~\eqref{eqn:intro/vary_w2/einstein} holds, then Corollary~\ref{cor:divE2} implies that $\csigma_{2,\phi}-\frac{1}{2\tau}\sigma_{1,\phi}$ is constant.  Since $\int(\psi-\frac{1}{2}\tr h)=0$ for variations within $\mC_1$, it follows from Proposition~\ref{prop:first_variation2} that $(g,\phi,\tau)$ is a critical point of $\mW_2\colon\mC_1\to\bR$.

Consider now the functional $\mW_2\colon\mC_2(g)\to\bR$.  As in the proof of Corollary~\ref{cor:critW1}, it suffices to consider variations of $(g,\phi)$ for which $h=cg$ and the constraint~\eqref{eqn:constraint} holds.  The conclusion is then an immediate consequence of Proposition~\ref{prop:first_variation2}.
\end{proof}
\section{Divergence structure for $\csigma_{k,\phi}$ on flat manifolds}
\label{sec:divergence}

In the conformal case, it is known that on a locally conformally flat manifold the divergence of the tracefree part of the $k$-th Newton tensor is a (nonzero) constant multiple of the derivative of the $\sigma_k$-curvature~\cite{Viaclovsky2000}.  Corollary~\ref{cor:divE1} and Corollary~\ref{cor:divE2} show that this is also true in the weighted setting, as together they imply that
\[ \delta_\phi\cE_{1,\phi} = -d\csigma_{1,\phi} \quad\text{and}\quad \delta_\phi\left(\cB_\phi + \cE_{2,\phi}\right) = -d\csigma_{2,\phi} \]
for all manifolds with density and all parameters $\lambda$.  Since the weighted Bach tensor $\cB_\phi$ vanishes on flat manifolds and since we stated in Section~\ref{sec:defn} that the trace-adjusted weighted $k$-th Newton tensors $\cE_\phi$ should be regarded as the weighted analogues of the tracefree parts of the $k$-th Newton tensors, we do indeed have that $\delta_\phi\cE_{k,\phi}=-d\csigma_{k,\phi}$ for all $k\in\{1,2\}$ on any flat manifold with density and with any parameter $\lambda$.  Unsurprisingly, the restriction $k\in\{1,2\}$ is unnecessary here; it is the purpose of this section to prove this fact.  Indeed, we compute the weighted divergence of the weighted Newton tensors.

To begin, we require the following formula for the derivative of the weighted $\sigma_k$-curvatures.

\begin{lem}
\label{lem:dsigmak}
Let $(M^n,g,e^{-\phi}\dvol)$ be a manifold with density and let $\lambda\in\bR$.  For any $k\in\bN$ it holds that
\begin{equation}
\label{eqn:dsigmak}
\nabla\csigma_{k,\phi} = \csigma_{k-1,\phi}\nabla\cY_\phi + \sum_{j=0}^{k-1}\frac{(-1)^j}{j+1}\csigma_{k-1-j,\phi}\nabla\tr\cRic_\phi^{j+1} .
\end{equation}
\end{lem}

\begin{proof}

The proof is by strong induction.  It is clear that~\eqref{eqn:dsigmak} holds when $k=1$.  Suppose then that~\eqref{eqn:dsigmak} holds for all $j\in\{1,\dotsc,k-1\}$.  From the definition of the weighted $\sigma_k$-curvature and the inductive hypothesis we see that
\begin{align*}
k\nabla\csigma_{k,\phi} & = \csigma_{k-1,\phi}\nabla\cY_\phi + \csigma_{k-2,\phi}\cY_\phi\nabla\cY_\phi + \sum_{j=0}^{k-2}\frac{(-1)^j}{j+1}\csigma_{k-2-j,\phi}\cY_\phi\nabla\tr\cRic_\phi^{j+1} \\
& \quad + \sum_{j=0}^{k-1}(-1)^j\csigma_{k-1-j,\phi}\nabla\tr\cRic_\phi^{j+1} + \sum_{j=0}^{k-2}(-1)^j\csigma_{k-2-j,\phi}\tr\cRic_\phi^{j+1}\nabla\cY_\phi \\
& \quad + \sum_{j=0}^{k-2}\sum_{\ell=0}^{k-2-j}\frac{(-1)^{j+\ell}}{\ell+1}\csigma_{k-2-j-\ell,\phi}\tr\cRic_\phi^{j+1}\nabla\cRic_\phi^{\ell+1} .
\end{align*}
The desired result then follows by interchanging the order of the last summation and using again the definition of the weighted $\sigma_k$-curvature.
\end{proof}

Viaclovsky showed~\cite{Viaclovsky2000} that the $k$-th Newton tensors are all divergence free on locally conformally flat manifolds, from which the claim about the divergence of their tracefree parts immediately follows.  It is straightforward to compute the divergence of the weighted $k$-th Newton tensors for arbitrary manifolds with density, and in particular to identify the contributions from the Riemann curvature tensor.

\begin{prop}
\label{prop:divTk}
Let $(M^n,g,e^{-\phi}\dvol)$ be a manifold with density and let $\lambda\in\bR$.  For any $k\in\bN$ it holds that
\begin{equation}
\label{eqn:divTk}
\delta_\phi\cT_{k,\phi} + \csigma_{k,\phi}\nabla\phi = \sum_{j=0}^{k-2}(-1)^j\cT_{k-2-j,\phi}\left(\cRic_\phi^{j+1}\cdot d\cRic_\phi\right)
\end{equation}
where, given sections $A,B$ of $S^2T^\ast M$, we define the vector field $A\cdot dB$ by
\[ \left\lp A\cdot dB, X\right\rp = \left\lp A, dB(\cdot,X,\cdot)\right\rp \]
for all vector fields $X$.
\end{prop}

\begin{proof}

The proof is by strong induction.  The case $k=1$ follows immediately from~\eqref{eqn:divE1}.  Suppose then that~\eqref{eqn:divTk} holds for all $1\leq j\leq k-1$.  Since
\begin{equation}
\label{eqn:cT_recursive}
\cT_{k,\phi} = \csigma_{k,\phi}g - \cT_{k-1,\phi}\cRic_\phi,
\end{equation}
where the second summand is to be regarded as a composition of endomorphisms of $TM$, we compute using the inductive hypothesis that
\begin{align*}
\delta_\phi\cT_{k,\phi} & = \nabla\csigma_{k,\phi} - \csigma_{k,\phi}\nabla\phi - \cT_{k-1,\phi}\cdot\nabla\cRic_\phi + \csigma_{k-1,\phi}\cRic_\phi(\nabla\phi) \\
& \quad - \sum_{j=0}^{k-3}(-1)^j\cRic_\phi\cT_{k-3-j,\phi}\left(\cRic_\phi^{j+1}\cdot d\cRic_\phi\right) ,
\end{align*}
where we define $\cT_{k-1,\phi}\cdot\nabla\cRic_\phi$ by the same contraction as $\cT_{k-1,\phi}\cdot d\cRic_\phi$.  By commuting derivatives and applying Lemma~\ref{lem:dsigmak} we find that
\[ \cT_{k-1,\phi} \cdot \nabla\cRic_\phi = \nabla\csigma_{k,\phi} - \csigma_{k-1,\phi}\nabla\cY_\phi + \cT_{k-1,\phi}\cdot d\cRic_\phi . \]
Using~\eqref{eqn:divE1} and~\eqref{eqn:cT_recursive}, we find that
\[ \cT_{k-1,\phi}\cdot d\cRic_\phi = \csigma_{k-1,\phi}\cRic_\phi(\nabla\phi) + \csigma_{k-1,\phi}\nabla\cY_\phi - \cT_{k-2,\phi}\cRic_\phi\cdot d\cRic_\phi . \]
Therefore
\begin{align*}
\delta_\phi\cT_{k,\phi} & = -\csigma_{k,\phi}\nabla\phi + \cT_{k-2,\phi}\cRic_\phi\cdot d\cRic_\phi - \sum_{j=0}^{k-3}(-1)^j\cRic_\phi\cT_{k-3-j,\phi}\left(\cRic_\phi^{j+1}\cdot d\cRic_\phi\right) \\
& = -\csigma_{k,\phi}\nabla\phi + \sum_{j=0}^{k-2}(-1)^j\cT_{k-2-j,\phi}\left(\cRic_\phi^{j+1}\cdot d\cRic_\phi\right) \\
& \quad + \left(\cT_{k-2,\phi}\cRic_\phi - \sum_{j=0}^{k-2}(-1)^j\csigma_{k-2-j,\phi}\cRic_\phi^{j+1}\right)\cdot d\cRic_\phi \\
& = -\csigma_{k,\phi}\nabla\phi + \sum_{j=0}^{k-2}(-1)^j\cT_{k-2-j,\phi}\left(\cRic_\phi^{j+1}\cdot d\cRic_\phi\right) ,
\end{align*}
as desired.
\end{proof}

As an immediate corollary of Proposition~\ref{prop:divTk}, we have the following analogue of Viaclovsky's result~\cite{Viaclovsky2000} on the divergence of the tracefree part of the $k$-th Newton tensor on a locally conformally flat manifold.

\begin{cor}
\label{cor:divEk}
Let $(M^n,g,e^{-\phi}\dvol)$ be a manifold with density and let $\lambda\in\bR$.  Suppose that $g$ is a flat metric.  Then for any $k\in\bN$ it holds that
\begin{equation}
\label{eqn:divEk}
\delta_\phi\cE_{k,\phi} = -d\csigma_{k,\phi} .
\end{equation}
\end{cor}

\begin{proof}

Since $g$ is flat, Proposition~\ref{prop:divTk} implies that
\[ \delta_\phi \cT_{k,\phi} + \csigma_{k,\phi}\nabla\phi = 0 . \]
The result then follows from the definition $\cE_{k,\phi}=\cT_{k,\phi} - \csigma_{k,\phi}g$ of the trace-adjusted $k$-th weighted Newton tensor.
\end{proof}

\begin{remark}
\label{rk:divergence}
All of the results in this section are not sensitive to the particular choice of constant in the definition of $\csigma_{k,\phi}$.  More precisely, for any constant $c\in\bR$, consider the curvature invariants
\[ \csigma_{k,\phi}^{(c)} := \sigma_k^\infty\left(\cY_\phi + c, \cRic_\phi\right) \]
and their associated trace-adjusted $k$-th weighted Newton tensors $\cE_{k,\phi}^{(c)}$.  Then Corollary~\ref{cor:divEk} is true for these invariants.  This is because the first step in the inductive arguments of Lemma~\ref{lem:dsigmak} and Proposition~\ref{prop:divTk} only relies on the fact
\[ \delta_\phi\cRic_\phi = d\left(\tr\cRic_\phi + \cY_\phi\right) = d\left(\tr\cRic_\phi + \cY_\phi + c\right) . \]
\end{remark}
\section{Obata theorems for the weighted $\sigma_k$-curvature}
\label{sec:obata}

Obata's theorem states that the only critical points of the Yamabe functional within the conformal class of an Einstein metric are themselves Einstein metrics~\cite{Obata1971}.  A more restrictive version of this result, proven by a similar technique, states that the only critical points in the positive elliptic $k$-cone of the volume-normalized total $\sigma_k$-curvature functional within the conformal class of the standard metric on $S^n$ are the Einstein metrics~\cite{Viaclovsky2000}.  In particular, this characterizes the critical points in the positive elliptic cone of the total $\sigma_k$-curvature functional on Euclidean space under strong restrictions on their decay at infinity; using the method of moving places, the decay assumptions at infinity can be removed~\cite{CGS1989,LiLi2003}.

The purpose of this section is to prove analogous Obata-type theorems for the weighted $\sigma_k$-curvature.  As motivation, we begin by recalling a result of Perelman~\cite{Perelman1} which states that the only critical points of the $\mW_1$-functional on a compact shrinking gradient Ricci soliton are themselves shrinking gradient Ricci solitons; see Theorem~\ref{thm:obata/1} for the precise statement.  Our first main result is Theorem~\ref{thm:obata/2}, which characterizes critical points in the negative weighted elliptic $2$-cone of the total weighted $\sigma_2$-curvature functional $\mW_2\colon\mC_1(dx^2)\to\bR$ on Euclidean space.  Our second result is Theorem~\ref{thm:obata/k}, which characterizes certain pairs $(\phi,\tau)\in\mC_1(dx^2)$ on Euclidean space such that $(\bR^n,dx^2,e^{-\phi}\dvol)$ has constant weighted $\sigma_k$-curvature and lies in the negative weighted elliptic $k$-cone.  This latter result can also be interpreted as a statement about critical points of a functional on $\mC_1(dx^2)$; see Section~\ref{sec:variational_status} for further discussion.

Both of our results require us to impose strong decay assumptions on the potential $\phi$ at infinity, and hence are best regarded as weighted analogues of the Obata-type theorems in~\cite{Viaclovsky2000}.  We expect that one can establish the same conclusion assuming only that $e^{-\phi}\dvol$ is a finite measure by combining our argument with an adaptation of the integral estimates developed in~\cite{ChangGurskyYang2003b,Gonzalez2006r} for the $\sigma_k$-curvatures.

\subsection{Critical points of the total weighted $\sigma_1$-curvature}
\label{subsec:obata/1}

In order to motivate our techniques, we recall the following simple form of Perelman's Obata-type theorem and its proof for critical points of the $\mW_1$-functional on compact shrinking gradient Ricci solitons~\cite{Perelman1}.  Carrillo and Ni~\cite{CarrilloNi2008} showed more generally that volume-normalized shrinking gradient Ricci solitons minimize the $\nu$-entropy.  It is expected, though to the best of the author's knowledge no such result appears in the literature, that the compactness assumption below can be removed.

\begin{thm}
\label{thm:obata/1}
Let $(M^n,g,e^{-\phi_0}\dvol)$ be a compact volume-normalized shrinking gradient Ricci soliton satisfying $\Ric_{\phi_0}=\frac{1}{2\tau_0}g$.  Suppose that $(\phi,\tau)\in\mC_1(g)$ is a critical point of the total weighted $\sigma_1$-curvature functional $\mW_1\colon\mC_1(g)\to\bR$.  Then $(\phi,\tau)=(\phi_0,\tau_0)$.
\end{thm}

\begin{proof}

Since $(M^n,g,e^{-\phi_0}\dvol)$ is a shrinking gradient Ricci soliton, it follows that the adjusted Bakry-\'Emery Ricci tensor $\cRic_\phi$ of $(M^n,g,e^{-\phi}\dvol)$ with parameter $\lambda=\frac{1}{2\tau}$ is
\begin{equation}
\label{eqn:cRic_confqe}
\cRic_\phi = \nabla^2\left(\phi - \phi_0\right) + \left(\lambda_0 - \lambda\right)g
\end{equation}
for $\lambda_0=\frac{1}{2\tau_0}$.  By Corollary~\ref{cor:critW1}, the assumption on $(\phi,\tau)$ implies that $\delta_\phi\cRic_\phi=0$ and $\int\tr\cRic_\phi=0$, where the integral is taken with respect to $e^{-\phi}\dvol$.  Hence
\begin{equation}
\label{eqn:obata/1/ibp}
0 = \int_M \lp\cRic_\phi,\nabla^2\left(\phi-\phi_0\right)+\left(\lambda_0-\lambda\right)g\rp\,e^{-\phi}\dvol = \int_M \lv\cRic_\phi\rv^2 e^{-\phi}\dvol .
\end{equation}
Thus $\cRic_\phi=0$.  The maximum principle applied to the trace of~\eqref{eqn:cRic_confqe} implies that $\lambda=\lambda_0$ and that $\phi-\phi_0$ is constant.  The desired result is then a consequence of the normalizations $(\phi_0,\tau_0),(\phi,\tau)\in\mC_1(g)$.
\end{proof}

\subsection{Critical points of the total weighted $\sigma_2$-curvature}
\label{subsec:obata/2}

We now turn to proving our classification of the shrinking Gaussians in terms of the $\mW_2$-functional.  There are two new concerns to address.  First, we need to deal with the noncompactness of Euclidean space; we deal with this by imposing strong decay assumptions on the potential $\phi$.  Second, we need to impose an ellipticity assumption to obtain a sign on the analogue of the integrand in~\eqref{eqn:obata/1/ibp}.

\begin{thm}
\label{thm:obata/2}
Suppose that $(\phi,\tau)\in \mC_1(dx^2)$ is a critical point of the total weighted $\sigma_2$-curvature functional $\mW_2\colon\mC_1(dx^2)\to\bR$ on Euclidean space.  Suppose additionally that
\begin{align*}
\csigma_{1,\phi} & < \frac{1}{2\tau}, \\
\csigma_{2,\phi} & > \frac{1}{2\tau}\csigma_{1,\phi} - \frac{1}{8\tau^2},
\end{align*}
and that the function $\cphi(x)=\lv x\rv^2\phi\left(x/\lv x\rv^2\right)$ extends to a $C^2$ function in $\bR^n$ such that $\cphi(0)>0$.  Then there is a point $x_0\in\bR^n$ such that
\begin{equation}
\label{eqn:obata/2}
\phi(x) = \frac{\lv x-x_0\rv^2}{4\tau} .
\end{equation}
\end{thm}

\begin{proof}

The adjusted Bakry-\'Emery Ricci tensor $\cRic_\phi$ of $(M^n,g,e^{-\phi}\dvol)$ with parameter $\lambda=\frac{1}{2\tau}$ is
\begin{equation}
\label{eqn:cRic_confqe2}
\cRic_\phi = \nabla^2\phi - \lambda g .
\end{equation}
Denote by
\begin{equation}
\label{eqn:hsigmak}
\hsigma_{k,\phi}:=\sigma_k^\infty\left(\cY_\phi-\lambda,\cRic_\phi\right)
\end{equation}
and by $\hE_{k,\phi}$ the associated trace-adjusted weighted $k$-th Newton tensors.  Then
\begin{align*}
\hsigma_{1,\phi} & = \csigma_{1,\phi} - \frac{1}{2\tau} \\
\hsigma_{2,\phi} & = \csigma_{2,\phi} - \frac{1}{2\tau}\csigma_{1,\phi} + \frac{1}{8\tau^2} \\
\hE_{2,\phi} & = \cE_{2,\phi} - \frac{1}{2\tau}\cE_{1,\phi} .
\end{align*}
Note that the sign assumptions on $\csigma_{k,\phi}$ are equivalent to $(-1)^k\hsigma_{k,\phi}>0$ for $k\in\{1,2\}$; i.e.\ $(M^n,g,e^{-\phi}\dvol)$ with parameter $\lambda=\frac{1}{2\tau}$ lies in the negative weighted elliptic $2$-cone $\hGamma_{2,\phi}^-$ determined by our definition of $\hsigma_{k,\phi}$.  By Theorem~\ref{thm:intro/vary_w2} and Corollary~\ref{cor:divE2}, we also have that $\delta_\phi\hE_{2,\phi}=0$ and $\int\tr\hE_{2,\phi}=0$.  Letting $B_s\subset\bR^n$ denote the ball of radius $s>0$ centered at the origin, we then compute that
\begin{equation}
\label{eqn:ibp}
\begin{split}
\lim_{s\to\infty}\int_{\partial B_s} \hE_{2,\phi}(\nabla\phi,\nabla s)e^{-\phi}\dvol & = \int_{\bR^n} \lp \hE_{2,\phi},\cRic_\phi\rp \,e^{-\phi}\dvol \\
& = \int_{\bR^n} \left[ 3\hsigma_{3,\phi} - \hsigma_{1,\phi}\hsigma_{2,\phi} \right] e^{-\phi}\dvol .
\end{split}
\end{equation}
The assumption $\cphi\in C^2$ with $\cphi(0)>0$ implies that there is a constant $c>0$ such that $\phi(x)\geq c\lv x\rv^2$ for $\lv x\rv$ sufficiently large and that $\lv\nabla\phi\rv$ and $\lv\nabla^2\phi\rv$ grow at most polynomially in $\lv x\rv$.  In particular, the boundary integral in~\eqref{eqn:ibp} vanishes in the limit $s\to\infty$.  It then follows from Corollary~\ref{cor:newton} that $(\bR^n,g,e^{-\phi}\dvol)$ is a shrinking gradient Ricci soliton, from which we readily compute that $\phi$ is of the form~\eqref{eqn:obata/2}.
\end{proof}

\begin{remark}

The decay assumption at infinity made in Theorem~\ref{thm:obata/2} is the direct analogue of the assumption made in~\cite{Viaclovsky2000b}.  It is clear from the proof that it suffices to assume that the boundary integral in~\eqref{eqn:ibp} vanishes as $s\to\infty$.  As mentioned at the beginning of this section, we expect that the decay assumption can be replaced by the assumption that $e^{-\phi}\dvol$ is a finite measure on $\bR^n$.
\end{remark}

\subsection{Constant $\sigma_k$-curvature}
\label{subsec:obata/k}

Finally, we consider simply-connected flat manifolds with density with for which the weighted $\sigma_k$-curvature $\hsigma_{k,\phi}$ defined by~\eqref{eqn:hsigmak} is constant for some positive parameter $\lambda$.  This is a natural generalization of Theorem~\ref{thm:obata/1} and Theorem~\ref{thm:obata/2}, though such manifolds with density are not generally critical points for the total weighted $\sigma_k$-curvature functional.  Nevertheless, they are critical points for some variational problem (cf.\ Section~\ref{sec:variational_status}).

\begin{thm}
\label{thm:obata/k}
Suppose that $(\phi,\tau)\in\mC_1(dx^2)$ is such that the manifold with density $(\bR^n,dx^2,e^{-\phi}\dvol)$ with parameter $\lambda=\frac{1}{2\tau}$ satisfies
\[ d\hsigma_{k,\phi} = 0 \quad\text{and}\quad \int_{\bR^n}\tr\hE_{k,\phi}\,e^{-\phi}\dvol = 0 \]
for some $k\in\bR$.  Suppose additionally that $(-1)^j\hsigma_{j,\phi}>0$ for all $j\in\{1,2,\dotsc,k\}$ and that the function $\cphi(x):=\lv x\rv^2\phi\left(x/\lv x\rv^2\right)$ extends to a $C^2$ function on $\bR^n$ with $\cphi(0)>0$.  Then there is a point $x_0\in\bR^n$ such that $\phi(x)=\frac{\lv x-x_0\rv^2}{4\tau}$.
\end{thm}

\begin{proof}

Note that $\cRic_\phi$ satisfies~\eqref{eqn:cRic_confqe2}.  As noted in Remark~\ref{rk:divergence}, we may apply Corollary~\ref{cor:divEk} to conclude that $\delta_\phi\hE_{k,\phi}=0$.  Using our decay assumptions on $\phi$, we may argue as in the proof of Theorem~\ref{thm:obata/2} that
\[ 0 = \int_{\bR^n} \lp \hE_{k,\phi}, \cRic_\phi\rp \,e^{-\phi}\dvol = \int_{\bR^n} \left[ (k+1)\hsigma_{k+1,\phi} - \hsigma_{1,\phi}\hsigma_{k,\phi} \right] e^{-\phi}\dvol . \]
From Corollary~\ref{cor:newton} it follows that $(\bR^n,g,e^{-\phi}\dvol)$ is a volume-normalized shrinking gradient Ricci soliton, from which the conclusion immediately follows.
\end{proof}
\section{The Variational Status of $\sigma_{k,\phi}$}
\label{sec:variational_status}

It is known that the $\sigma_k$-curvature is variational in a conformal class $[g]$ if and only if $k\in\{1,2\}$ or $[g]$ is a class of locally conformally flat metrics~\cite{BransonGover2008}, and moreover, in the variational cases the $\sigma_k$-curvature problem can be solved~\cite{ShengTrudingerWang2007}.  The purpose of this section is to show that the first statement holds in the weighted sense.

\begin{thm}
\label{thm:variational_status}
Let $(M^n,g)$ be a Riemannian manifold and fix $\lambda\in\bR$ and a positive integer $k\leq n$.  Then the weighted $\sigma_k$-curvature $\csigma_{k,\phi}$ is variational if and only if $k\in\{1,2\}$ or the Riemann curvature tensor $\Rm$ of $g$ vanishes identically.
\end{thm}

The proof of this result is completely analogous to the proof given by Branson and Gover~\cite{BransonGover2008} of the corresponding result in conformal geometry.  For technical convenience, we assume throughout this section that our manifolds are compact; since all statements are local statements this is not a meaningful restriction.

To begin, we make precise what it means for the weighted $\sigma_k$-curvature to be variational (cf.\ \cite{BransonGover2008}).  A \emph{local scalar invariant} is a scalar function defined on a manifold with density $(M^n,g,e^{-\phi}\dvol)$ which is built polynomially from $g$ and its inverse, the Levi-Civita connection, the Riemann curvature tensor, and $\phi$.  In particular, the weighted $\sigma_k$-curvatures $\csigma_{k,\phi}$ are local scalar invariants for all $k\in\bN$ and all $\lambda\in\bR$.  A local scalar invariant $L_\phi$ is \emph{variational on a Riemannian manifold $(M^n,g)$} if there is a functional $\mF\colon C^\infty(M)\to\bR$ such that
\[ \mF^\bullet(\phi)[\psi] = \int_M \psi L_\phi\,e^{-\phi}\dvol \]
for all $\phi,\psi\in C^\infty(M)$, where
\[ \mF^\bullet(\phi)[\psi] := \left.\frac{d}{dt}\right|_{t=0}\mF(\phi+t\psi) \]
as in Section~\ref{sec:first_variation}.

The following lemma provides a way to characterize variational local scalar invariants in terms of their linearizations (cf.\ \cite[Lemma~2]{BransonGover2008}).

\begin{lem}
\label{lem:variational_characterization}
Let $L_\phi$ be a local scalar invariant and let $D_\phi=L_\phi^\bullet$ be its linearization.
\begin{enumerate}
\item $L_\phi$ is variational on a Riemannian manifold $(M^n,g)$ if and only if $D_\phi$ is formally self-adjoint at all functions $\phi\in C^\infty(M)$.
\item If $L_\phi$ is variational then
\begin{equation}
\label{eqn:total_bullet}
\left(\int_M L_\phi\,e^{-\phi}\dvol\right)^\bullet[\psi] = -\int_M \left(L_\phi - D_\phi(1)\right)\psi\,e^{-\phi}\dvol .
\end{equation} 
\end{enumerate}
\end{lem}

\begin{proof}

The proof of the first statement is essentially the same as the proof of~\cite[Lemma~2]{BransonGover2008}, though we shall present it here for convenience.  Since $C^\infty(M)$ is contractible, we see that $L_\phi$ is variational on $(M^n,g)$ if and only if the putative second variation
\begin{equation}
\label{eqn:mS}
\mS(\eta,\omega) := \int_M \eta\left( L_\phi e^{-\phi}\dvol\right)^\bullet = \int_M \eta\left(D_\phi - L_\phi\right)\omega\,e^{-\phi}\dvol
\end{equation}
is symmetric at all functions $\phi$.  From the right-hand side of~\eqref{eqn:mS} it is clear that this holds if and only if $D_\phi$ is formally self-adjoint.

To see the second statement, simply observe that
\[ \left(\int_M L_\phi\,e^{-\phi}\dvol\right)^\bullet[\psi] = \mS(1,\psi) . \]
Since $L_\phi$ is variational, $D_\phi$ is formally self-adjoint, and the result then follows immediately from~\eqref{eqn:mS}.
\end{proof}

From Lemma~\ref{lem:variational_characterization} we see that the first step in proving Theorem~\ref{thm:variational_status} is to compute the linearization of $\csigma_{k,\phi}$ for each $k$.

\begin{prop}
\label{prop:diff}
Let $(M^n,g,e^{-\phi}\dvol)$ be a manifold with density, let $\lambda\in\bR$, and let $k\in\bN$.  The linearization $\cD_{k,\phi}$ of $\csigma_{k,\phi}$ is
\begin{equation}
\label{eqn:Dk}
\cD_{k,\phi}(\psi) := \left\lp\cT_{k-1,\phi},\nabla^2\psi\right\rp - \csigma_{k-1,\phi}\lp\nabla\phi,\nabla\psi\rp + \lambda\csigma_{k-1,\phi}\psi .
\end{equation}
\end{prop}

\begin{proof}

Our proof is by strong induction.  Recalling that we are linearizing in $\phi$ only, we observe that
\begin{equation}
\label{eqn:diff_blocks}
\cRic_\phi^\bullet = \nabla^2\psi \quad\text{and}\quad \cY_\phi^\bullet = -\lp\nabla\phi,\nabla\psi\rp + \lambda\psi .
\end{equation}
It follows readily from~\eqref{eqn:diff_blocks} that $\csigma_{1,\phi}^\bullet=\left(\Delta_\phi+\lambda\right)\psi$; i.e.\ \eqref{eqn:Dk} is valid for $k=1$.

Suppose now that~\eqref{eqn:Dk} is valid for $\cD_{j,\phi}$ for $1\leq j\leq k-1$.  It follows from the definitions of the weighted $\sigma_k$-curvature and the weighted Newton tensor that
\begin{align*}
\sum_{j=0}^{k-2}(-1)^j\csigma_{k-2-j,\phi}\tr P^{j+1} & = (k-1)\csigma_{k-1,\phi} - \cY_\phi\csigma_{k-2,\phi} \\
\sum_{j=0}^{k-2}(-1)^j \cT_{k-2-j,\phi}\tr P^{j+1} & = \sum_{j=0}^{k-1}(-1)^j\left[(k-1-j)\csigma_{k-1-j,\phi} - \cY_\phi\csigma_{k-2-j,\phi}\right]\cRic_\phi^j .
\end{align*}
On the other hand, the inductive hypothesis implies that
\begin{align*}
k\csigma_{k,\phi}^\bullet & = \cY_\phi\left\lp\cT_{k-2,\phi},\nabla^2\psi\right\rp - \cY_\phi\csigma_{k-2,\phi}\lp\nabla\phi,\nabla\psi\rp + \lambda\cY_\phi\csigma_{k-2,\phi}\psi \\
& \quad - \csigma_{k-1,\phi}\lp\nabla\phi,\nabla\psi\rp + \lambda\csigma_{k-1,\phi}\psi + \sum_{j=0}^{k-1}(-1)^j(j+1)\csigma_{k-1-j,\phi}\left\lp\cRic_\phi^j,\nabla^2\psi\right\rp \\
& \quad + \sum_{j=0}^{k-2}(-1)^j\left[\lp \cT_{k-2-j,\phi},\nabla^2\psi\rp - \csigma_{k-2-j,\phi}\lp\nabla\phi,\nabla\psi\rp + \lambda\csigma_{k-2-j,\phi}\psi\right]\tr P^{j+1} .
\end{align*}
Combining these formulae yields the desired result.
\end{proof}

As an immediate consequence of Proposition~\ref{prop:divTk} and Proposition~\ref{prop:diff} we have the following equivalent formulation of the linearization of $\csigma_{k,\phi}$.

\begin{cor}
\label{cor:Dk}
Let $(M^n,g,e^{-\phi}\dvol)$ be a manifold with density, let $\lambda\in\bR$, and let $k\in\bN$.  The linearization $\cD_{k,\phi}$ of $\csigma_{k,\phi}$ is
\begin{equation}
\label{eqn:Dk_sa}
\begin{split}
\cD_{k,\phi}(\psi) & = \delta_\phi\left(\cT_{k-1,\phi}(\nabla\psi)\right) + \lambda\csigma_{k-1,\phi}\psi \\
& \quad - \sum_{j=0}^{k-3} (-1)^j\cT_{k-3-j,\phi}\left(\cRic_\phi^{j+1}\cdot d\cRic_\phi,\nabla\psi\right) ,
\end{split}
\end{equation}
where we interpret the empty sum to equal zero.
\end{cor}

From~\eqref{eqn:Dk_sa} it is clear that $\cD_{k,\phi}$ is formally self-adjoint if and only if the last summation is formally self-adjoint.  As we prove below, this holds if and only if the manifold is flat.

\begin{lem}
\label{lem:to_flat}
Let $(M^n,g)$ be a Riemannian manifold and fix a parameter $\lambda\in\bR$ and an integer $k\in\bN$.  For each $\phi\in C^\infty(M)$, define
\begin{equation}
\label{eqn:cS}
\cS_{k,\phi} := \sum_{j=0}^{k-3} (-1)^j\cT_{k-3-j,\phi}\left(\cRic_\phi^{j+1}\cdot d\cRic_\phi\right),
\end{equation}
where all tensors are determined by the manifold with density $(M^n,g,e^{-\phi}\dvol)$ and the parameter $\lambda$.  Then $\cS_{k,\phi}=0$ for all $\phi\in C^\infty(M)$ if and only if $k\in\{1,2\}$ or $g$ is flat.
\end{lem}

\begin{remark}

This argument is slightly different than the one given in~\cite{BransonGover2008}.  The reason for the difference is that $\csigma_{k,\phi}$ is in general nontrivial when $k>n$, while the natural generalization of the argument given in~\cite{BransonGover2008} only proves Lemma~\ref{lem:to_flat} when $k\leq n$.
\end{remark}

\begin{proof}[Proof of Lemma~\ref{lem:to_flat}]

Clearly $\cS_{k,\phi}=0$ if $k\in\{1,2\}$.  As pointed out in Section~\ref{sec:smms}, if $g$ is flat then $d\cRic_\phi=0$, and hence $\cS_{k,\phi}=0$, for all $\phi\in C^\infty(M)$.  Therefore it suffices to show that if $k\geq 3$ and $\cS_{k,\phi}=0$ for all $\phi\in C^\infty(M)$, then $g$ is flat.

To that end, fix $p\in M$ and $\phi\in C^\infty(M)$.  Let $X_p\in T_pM$ be such that $\lv X_p\rv^2=2$.  Choose a function $\psi\in C^\infty(M)$ such that $\psi(p)=0$, $\lv\nabla\psi\rv^2(p)=X_p$ and $\nabla^2\psi\rv_p=0$.  Set $\phi_t=\phi+t\psi$.  The choice of $\psi$ implies that
\begin{align*}
\cY_{\phi_t}(p) & = \cY_\phi(p) - t\lp\nabla\phi\rv_p,X_p\rp - t^2 , \\
\cRic_{\phi_t}\rv_p & = \cRic_\phi\rv_p, \\
d\cRic_{\phi_t}\rv_p & = d\cRic_\phi\rv_p - t\Rm(\cdot,\cdot,X_p,\cdot) .
\end{align*}
It is clear that $\cS_{k,\phi_t}\rv_p$ is a $T_pM$-valued polynomial in $t$ of degree $2k-5$.  By Proposition~\ref{prop:ordinary_sigmak} we see that
\[ \cS_{k,\phi_t}\rv_p = (-1)^{k}\frac{t^{2k-5}}{(k-3)!}\cRic_\phi\cdot\Rm(\cdot,\cdot,X_p,\cdot)\rv_p + \mbox{l.o.t.} \]
Since $p$ and $X_p$ are arbitrary, this implies that $\cRic_\phi\cdot\Rm=0$ for all $\phi\in C^\infty(M)$, and hence $\nabla^2\phi\cdot\Rm=0$ for all $\phi\in C^\infty(M)$.  Arguing as above --- namely by letting $\nabla^2\phi\rv_p = X_p\otimes X_p$ for $X_p\in T_pM$ arbitrary --- shows that $\Rm\equiv 0$, as desired.
\end{proof}

The proof of Theorem~\ref{thm:variational_status} now amounts to showing that the last summation in~\eqref{eqn:Dk_sa} vanishes for all $\phi\in C^\infty(M)$.

\begin{proof}[Proof of Theorem~\ref{thm:variational_status}]

We argue as in the proof of~\cite[Theorem~1]{BransonGover2008}.  From Lemma~\ref{lem:variational_characterization} and Corollary~\ref{cor:Dk} we see that $\csigma_{k,\phi}$ is variational if and only if for each $\phi\in C^\infty(M)$, the differential operator $\cD_{k,\phi}$ is formally self-adjoint with respect to $e^{-\phi}\dvol$.  Hence $\csigma_{k,\phi}$ is variational if and only if
\begin{equation}
\label{eqn:Sk_sa}
\int_M \left\lp \cS_{k,\phi}, \eta\nabla\omega-\omega\nabla\eta\right\rp\,e^{-\phi}\dvol = 0
\end{equation}
for all $\eta,\omega\in C^\infty(M)$ and for $\cS_{k,\phi}$ defined in~\eqref{eqn:cS}.  Taking $\eta\equiv1$ and $\omega$ arbitrary shows that~\eqref{eqn:Sk_sa} holds if and only if $\delta_\phi \cS_{k,\phi}=0$.  Hence~\eqref{eqn:Sk_sa} holds if and only if
\[ \int_M \left\lp \cS_{k,\phi},\eta\nabla\omega\right\rp\,e^{-\phi}\dvol = 0 \]
for all $\eta,\omega\in C^\infty(M)$, or equivalently, if and only if $\cS_{k,\phi}=0$.  That is, we have just shown that $\csigma_{k,\phi}$ is variational if and only if $\cS_{k,\phi}=0$ for each $\phi\in C^\infty(M)$.  The desired conclusion now follows from Lemma~\ref{lem:to_flat}.
\end{proof}

We conclude with a few words about how to find functionals $\hmW_k\colon\mC_1(g)\to\bR$ whose critical points have $\hsigma_{k,\phi}$ constant.  If $k\in\{1,2\}$, this is easily accomplished using Proposition~\ref{prop:first_variation1} and Proposition~\ref{prop:first_variation2}; i.e.\ simply take $\hmW_k=\mW_k$ in these cases.  When $k\geq 3$, Theorem~\ref{thm:variational_status} implies that we must assume that $(M^n,g)$ is flat, and then Lemma~\ref{lem:variational_characterization} and Proposition~\ref{prop:diff} imply that
\[ \left(\int_M \csigma_{k,\phi}e^{-\phi}\dvol\right)^\bullet = -\int_M \left(\csigma_{k,\phi} - \lambda\csigma_{k-1,\phi}\right)\psi\,e^{-\phi}\dvol \]
through variations of $\phi$.  From this it is easy to write down a linear combination $\hmW_k$ of the total weighted $\sigma_j$-curvature functionals with $j\in\{1,\dotsc,k\}$ such that
\[ \hmW_k(\phi,\tau)^\bullet = -\int_M \tau^k\hsigma_{k,\phi}\psi\,(4\pi\tau)^{-\frac{n}{2}}e^{-\phi}\dvol . \]
\section{Gradient Ricci Solitons as Local Extrema}
\label{sec:local_extrema}

Complementing the global classification of critical points of the $\hmW_k$-functionals in Section~\ref{sec:obata}, we show in this section that shrinking gradient Ricci solitons locally minimize (resp.\ maximize) the $\hmW_k$-functional for $k$ odd (resp.\ even) in the cases when the weighted $\sigma_k$-curvature is variational.  Indeed, we show that shrinking gradient Ricci solitons are strict local extrema except when they factor through a shrinking Gaussian.  Moreover, in the latter cases the only nondegeneracies come from the freedom to choose the basepoint of a shrinking Gaussian or, when the manifold is a shrinking Gaussian, from the scale-invariance of the shrinking Gaussians.  As throughout this article, these results and their proofs are close analogues of similar results of Viaclovsky in the conformal setting~\cite{Viaclovsky2000}.

\subsection{Local extrema of $\mW_1$}
\label{subsec:local_extrema/1}

In order to illustrate our techniques, we first show that shrinking gradient Ricci solitons are local extrema for the $\mW_1$-functional.  Of course, Theorem~\ref{thm:obata/1} and Perelman's proof~\cite{Perelman1} that minimizers of the $\nu$-entropy exist on compact shrinking gradient Ricci solitons proves that they are global extrema (cf.\ \cite{CarrilloNi2008}), but the technique used to prove the local result is also applicable to the total weighted $\sigma_k$-curvature functionals.  The proof of the local result requires the following weighted analogue of the Lichnerowicz--Obata theorem proven by Cheng and Zhou~\cite{ChengZhou2013}.

\begin{prop}
\label{prop:weighted_lichnerowicz}
Let $(M^n,g,e^{-\phi}\dvol)$ be a Riemannian manifold with density and suppose that $\Ric_\phi\geq\lambda g>0$.  Then $\lambda_1(-\Delta_\phi)\geq\lambda$ for
\[ \lambda_1\left(\Delta_\phi\right) := \inf_{u\in C_0^\infty(M)\setminus\{0\}}\left\{ \frac{\int_M \lv\nabla u\rv^2e^{-\phi}\dvol}{\int_M u^2e^{-\phi}\dvol} \colon \int_M ue^{-\phi}\dvol = 0 \right\} \]
the first nonzero eigenvalue of the weighted Laplacian.  Moreover, if $\lambda_1(-\Delta_\phi)=\lambda$ then there is a manifold with density $(N^{n-1},h,e^{-\psi}\dvol)$ such that $\Ric_\psi\geq\lambda h$ and $(M^n,g)$ is isometric to $(\bR\times N^{n-1},dx^2\oplus h)$ and $\phi=\frac{\lambda}{2}x^2 + \psi$.
\end{prop}

In fact, Cheng and Zhou~\cite{ChengZhou2013} show that if $\lambda$ is an eigenvalue of $-\Delta_\phi$ of multiplicity $k\geq1$ for $(M^n,g,e^{-\phi}\dvol)$ as in Proposition~\ref{prop:weighted_lichnerowicz}, then $k\leq n$ and $(M^n,g,e^{-\phi}\dvol)$ splits isometrically as a product with a $k$-dimensional shrinking Gaussian.

We can now show that shrinking gradient Ricci solitons are local minima of the $\mW_1$-functional.

\begin{thm}
\label{thm:local_extrema_W1}
Let $(M^n,g,e^{-\phi}\dvol)$ be a volume-normalized shrinking gradient Ricci soliton with $\Ric_\phi=\frac{1}{2\tau}g$.  Let $(\phi_t,\tau_t)\in\mC_1(g)$ be a smooth variation of $(g,\tau)$.  Then
\begin{equation}
\label{eqn:local_extrema_W1}
\left.\frac{d^2}{dt^2}\mW_1\left(g,\phi_t,\tau_t\right)\right|_{t=0} \geq 0 .
\end{equation}
Moreover, if $(M^n,g,e^{-\phi}\dvol)$ does not factor through a shrinking Gaussian, then the inequality~\eqref{eqn:local_extrema_W1} is strict for nontrivial variations.
\end{thm}

\begin{remark}
\label{rk:local_extrema_W1}
In fact, equality holds in~\eqref{eqn:local_extrema_W1} if and only if either $(M^n,g,e^{-\phi}\dvol)$ is a shrinking Gaussian and the variation $(\phi_t,\tau_t)$ is through shrinking Gaussians or if $(M^n,g,e^{-\phi}\dvol)$ is isometric to a product of a shrinking Gaussian and a correspondingly-normalized shrinking gradient Ricci soliton and the variation $(\phi_t,\tau_t)$ has $\tau_t$ fixed and $\phi_t$ varying only through the choice of basepoint in the Gaussian factor, as is easily concluded from the proof below.
\end{remark}

\begin{proof}

Lemma~\ref{lem:modified_sigmak_grs} and Corollary~\ref{cor:critW1} together imply that $(\phi,\tau)$ is a critical point of $\mW_1\colon\mC_1(g)\to\bR$, and hence we may compute the left-hand side of~\eqref{eqn:local_extrema_W1} by restricting to paths $(\phi_t,\tau_t)\in\mC_1(g)$ of the form $\phi_t=\phi+t\psi$ and $\tau_t=\tau+\alpha t$ for some $\psi\in C^\infty(M)$ and some $\alpha \in\bR$.  Moreover, from Corollary~\ref{cor:Dk} it is easy to see that
\begin{equation}
\label{eqn:ddt_csigma1}
\frac{d}{dt}\left(\tau_t\csigma_{1,\phi_t}\right) = \left(\tau_t\Delta_{\phi_t}+\frac{1}{2}\right)\left(\psi+\frac{n\alpha}{2\tau_t}\right) + \alpha\left(\sigma_{1,\phi_t}-\frac{n}{4\tau_t}\right) .
\end{equation}
It therefore follows that
\[ \frac{d}{dt}\mW_1(g,\phi_t,\tau_t) = \int_M \left[-\tau_t\csigma_{1,\phi_t} + \frac{1}{2}\right]\left(\psi+\frac{n\alpha}{2\tau_t}\right) + \alpha\left(\sigma_{1,\phi_t}-\frac{n}{4\tau_t}\right) , \]
where the integral is taken with respect to $(4\pi\tau_t)e^{-\phi_t}\dvol$ (cf.\ Proposition~\ref{prop:first_variation1}).  Since $(\phi,\tau)$ is a critical point of $\mW_1\colon\mC_1(g)\to\bR$, it follows from the above display that
\begin{align*}
\left.\frac{d^2}{dt^2}\mW_1(g,\phi_t,\tau_t)\right|_{t=0} & = \int_M \left(\psi+\frac{n\alpha}{2\tau_t}\right)\left.\frac{d}{dt}\left(-\tau_t\csigma_{1,\phi_t}\right)\right|_{t=0} \\
& \quad + \int_M \frac{n\alpha^2}{4\tau^2} - \alpha\left(\sigma_{1,\phi}-\frac{n}{4\tau}\right)\psi_0 \\
& = \int_M \tau\lv\nabla\psi_0\rv^2 - \frac{1}{2}\psi_0^2 - 2\alpha\psi_0\left(\sigma_{1,\phi}-\frac{n}{4\tau}\right) + \frac{n\alpha^2}{4\tau^2} ,
\end{align*}
where $\psi_0:=\psi+\frac{n\alpha}{2\tau}$ and all integrals are taken with respect to $(4\pi\tau)^{-\frac{n}{2}}e^{-\phi}\dvol$.  We simplify the above display using two facts.  First, since $\Ric_\phi=\frac{1}{2\tau}g$, it holds that $\sigma_{1,\phi}-\frac{n}{4\tau}=\frac{1}{2}\Delta_\phi\phi$.  Second, we may find a constant $c\in\bR$ such that $\psi_0=\psi_1+c\phi_0$ for $\phi_0$ as in Lemma~\ref{lem:grs_potential} and $\psi_1\in C^\infty(M)$ such that $\int \psi_1\phi_0=0$.  In particular, this implies that
\[ \int_M \lv\nabla\psi_0\rv^2 = \int_M \lv\nabla\psi_1\rv^2 + c^2\lv\nabla\phi_0\rv^2 \quad\text{and}\quad \int_M \psi_0^2 = \int_M \psi_1^2 + c^2\phi_0^2 . \]
Using Lemma~\ref{lem:grs_potential} we compute that
\begin{align*}
\left.\frac{d^2}{dt^2}\mW_1(g,\phi_t,\tau_t)\right|_{t=0} & = \int_M \tau\lv\nabla\psi_1\rv^2 - \frac{1}{2}\psi_1^2 + \frac{c^2}{2}\phi_0^2 + \frac{\alpha c}{\tau}\phi_0^2 + \frac{n\alpha^2}{4\tau^2} \\
& \geq \int_M \tau\lv\nabla\psi_1\rv^2 - \frac{1}{2}\psi_1^2 + \frac{1}{2}\left(c+\frac{\alpha}{\tau}\right)^2\phi_0^2
\end{align*}
with equality if and only if $(M^n,g,e^{-\phi}\dvol)$ is isometric to a shrinking Gaussian.  It is then an immediate consequence of Proposition~\ref{prop:weighted_lichnerowicz} that~\eqref{eqn:local_extrema_W1} holds, and moreover, that the inequality is strict if $(M^n,g,e^{-\phi}\dvol)$ does not factor through a shrinking Gaussian.
\end{proof}

\subsection{Local extrema of $\mW_2$}
\label{subsec:local_extrema/2}

Lemma~\ref{lem:grs_potential} and Proposition~\ref{prop:weighted_lichnerowicz} also provide the main observations necessary to prove Theorem~\ref{thm:intro/local_extrema_W2} from the introduction, which we restate here for the convenience of the reader.

\begin{thm}
\label{thm:local_extrema_W2}
Let $(M^n,g,e^{-\phi}\dvol)$ be a volume-normalized shrinking gradient Ricci soliton with $\Ric_\phi=\frac{1}{2\tau}g$.  Let $(\phi_t,\tau_t)\in\mC_1(g)$ be a smooth variation of $(\phi,\tau)$.  Then
\begin{equation}
\label{eqn:local_extrema_W2}
\left.\frac{d^2}{dt^2}\mW_2\left(g,\phi_t,\tau_t\right)\right|_{t=0} \leq 0 .
\end{equation}
Moreover, equality holds in~\eqref{eqn:local_extrema_W2} for nontrivial variations $(\phi_t,\tau_t)\in\mC_1(g)$ if and only if $(M^n,g,e^{-\phi}\dvol)$ factors as an isometric product with a shrinking Gaussian and the variation is tangent to a curve through shrinking Gaussians in the Euclidean factor.
\end{thm}

\begin{proof}

As in the proof of Theorem~\ref{thm:local_extrema_W1}, we may suppose that $\phi_t=\phi+t\psi$ and that $\tau_t=\tau+\alpha t$ for some $\psi\in C^\infty(M)$ and some $t\in\bR$.  Using Corollary~\ref{cor:Dk} it is easy to compute that
\begin{equation}
\label{eqn:ddt_csigma2}
\begin{split}
\frac{d}{dt}\left(\tau_t^2\csigma_{2,\phi_t}\right) & = \tau_t^2\delta_{\phi_t}\left(\cT_{1,\phi_t}(\nabla\psi)\right) + \frac{\tau_t}{2}\csigma_{1,\phi_t}\left(\psi+\frac{n\alpha}{2\tau_t}\right) + 2\alpha\tau_t\sigma_{2,\phi_t} \\
& \quad + \frac{\alpha}{2}\left((\phi_t-n+2)\sigma_{1,\phi_t} - \Delta_{\phi_t}\phi_t - \frac{n}{2}\csigma_{1,\phi_t}\right) .
\end{split}
\end{equation}
In particular, it holds that
\begin{align*}
\frac{d}{dt}\mW_2(g,\phi_t,\tau_t) & = \int_M \left(-\tau^2\csigma_{2,\phi_t} + \frac{\tau_t}{2}\csigma_{1,\phi_t}\right)\left(\psi + \frac{n\alpha}{2\tau_t}\right) \\
& \quad + \alpha\int_M \left(2\tau_t\sigma_{2,\phi_t} + \frac{1}{2}\left(\phi_t-n+2\right)\sigma_{1,\phi_t} - \frac{n}{4}\csigma_{1,\phi_t}\right) .
\end{align*}
Differentiating again and using the assumption that $\Ric_\phi=\frac{1}{2\tau}g$ yields
\begin{align*}
\left.\frac{d^2}{dt^2}\mW_2(g,\phi_t,\tau_t)\right|_{t=0} & = \int_M \left[ \tau^2\cT_{1,\phi} - \frac{\tau}{2}g\right](\nabla\psi_0,\nabla\psi_0) - \frac{\tau}{2}\left(\csigma_{1,\phi}-\frac{1}{2\tau}\right)\psi_0^2 \\
& \quad - 2\alpha \int_M \left[ 2\tau\sigma_{2,\phi} + \frac{1}{2}(\phi-n+2)\sigma_{1,\phi} - \frac{1}{2}\sigma_{1,\phi} - \frac{1}{2}\Delta_\phi\phi\right]\psi_0 \\
& \quad + \alpha^2\int_M \left[ 2\sigma_{2,\phi} + \frac{n}{8\tau^2}(\phi-n) - \frac{n}{4\tau}\sigma_{1,\phi} + \frac{n^2}{16\tau^2} \right]
\end{align*}
for $\psi_0:=\psi+\frac{n\alpha}{2\tau}$; note that we have made some simplifications by using the fact that $\int_M\psi_0=0$.  We simplify this with two computations.  First, it is straightforward to verify using~\eqref{eqn:unmodified_sigma1} and~\eqref{eqn:unmodified_sigma2} that
\begin{align*}
2\tau\sigma_{2,\phi} + \frac{1}{2}\left(\phi-n+2\right)\sigma_{1,\phi} - \frac{1}{2}\Delta_\phi\phi - \frac{1}{2}\sigma_{1,\phi} = -\frac{1}{2}\left(\csigma_{1,\phi}-\frac{1}{2\tau}\right)\left(\phi_0 - \frac{n}{2}\right) .
\end{align*}
Second, using Lemma~\ref{lem:total_sigmak_grs} we compute that
\begin{align*}
\int_M \left[ 2\sigma_{2,\phi} + \frac{n}{8\tau^2}(\phi-n) - \frac{n}{4\tau}\sigma_{1,\phi} + \frac{n^2}{16\tau^2} \right] & = \int_M \left[ \frac{1}{4\tau^2}\phi_0^2 - \frac{n}{4\tau^2} + \frac{n}{4\tau}\csigma_{1,\phi} \right] \\
& \leq \frac{1}{2\tau}\left(\csigma_{1,\phi} - \frac{1}{2\tau}\right)\int_M \phi_0^2,
\end{align*}
where the last inequality follows from Lemma~\ref{lem:modified_sigmak_grs} and Lemma~\ref{lem:grs_potential}.  Writing $\psi_0=\psi_1+c\phi_0$ as in the proof of Theorem~\ref{thm:local_extrema_W1}, it follows that
\[ \left.\frac{d^2}{dt^2}\mW_2(g,\phi_t,\tau_t)\right|_{t=0} \leq \tau^2\left(\csigma_{1,\phi_t}-\frac{1}{2\tau}\right)\int_M \left[ \lv\nabla\psi_1\rv^2 - \frac{1}{2\tau}\psi_1^2 + \frac{1}{2\tau}\left(c+\frac{\alpha}{\tau}\right)^2\phi_0^2 \right] . \]
By Lemma~\ref{lem:modified_sigmak_grs} and Proposition~\ref{prop:weighted_lichnerowicz} we see that this is nonpositive and vanishes for nontrivial deformations if and only if (a) $(M^n,g,e^{-\phi}\dvol)$ factors through a shrinking Gaussian, $\alpha=0$, and $-\Delta_\phi\psi_0=\frac{1}{2\tau}\psi_0$; or (b) $(M^n,g,e^{-\phi}\dvol)$ is a shrinking Gaussian, $c+\frac{\alpha}{\tau}=0$, and $-\Delta_\phi\psi_1=\frac{1}{2\tau}\psi_1$.
\end{proof}

\subsection{Local extrema and constant $\hsigma_{k,\phi}$}
\label{subsec:local_extrema/k}

For similar reasons as in Subsection~\ref{subsec:obata/k}, we show in this subsection that shrinking gradient Ricci solitons on flat manifolds are local extrema for the functionals $\hmW_k\colon\mC_1(dx^2)\to\bR$ described at the end of Section~\ref{sec:variational_status}.  Since the only shrinking gradient Ricci solitons on flat manifolds are the shrinking Gaussians, the global classification is a consequence of Theorem~\ref{thm:obata/k}, and hence the result of this subsection is that the shrinking Gaussians minimize (resp.\ maximize) the $\hmW_k$-functional within the negative weighted elliptic $k$-cone $\hGamma_k^{\infty,-}$ when $k$ is odd (resp.\ even).

\begin{thm}
\label{thm:local_extrema_Wk}
Let $(\bR^n,dx^2,e^{-\phi}\dvol)$ be a shrinking Gaussian with $\Ric_\phi=\frac{1}{2\tau}g$.  Let $(\phi_t,\tau_t)\in\mC_1(dx^2)$ be a smooth variation of $(\phi,\tau)$.  Then
\begin{equation}
\label{eqn:local_extrema_Wk}
(-1)^k\left.\frac{d^2}{dt^2}\hmW_k(dx^2,\phi_t,\tau_t)\right|_{t=0} \leq 0 .
\end{equation}
Moreover, equality holds if and only if $(\phi_t,\tau_t)$ is tangent to a curve in $\mC_1(dx^2)$ through shrinking Gaussians.
\end{thm}

\begin{proof}

Without loss of generality, we may suppose that $\tau_t=\tau$.  This is because we can take a one-parameter family of isometries $\Phi\colon\bR^n\to\bR^n$ such that $\tau_t^{-1}\tau\Phi_t^\ast dx^2 = dx^2$, whence Lemma~\ref{lem:mW_scale_invariance} and the diffeomorphism invariance of the $\hmW_k$-functionals imply that
\[ \hmW_k\left(dx^2,\phi_t,\tau_t\right) = \hmW_k\left(\tau_t^{-1}\tau\,dx^2,\phi_t,\tau\right) = \hmW_k\left(dx^2,\Phi_t^\ast\phi_t,\tau\right) . \]
Furthermore, since $(\phi_t,\tau)$ is a critical point of the $\hmW_k$-functional, we may suppose that $\phi_t=\phi+t\psi$.  By the construction of $\hmW_k$, it holds that
\[ \frac{d}{dt}\hmW_k(dx^2,\phi_t,\tau) = -\int_{\bR^n} \tau^k\hsigma_{k,\phi}\psi\,(4\pi\tau)^{-\frac{n}{2}}e^{-\phi}\dvol . \]
Using Corollary~\ref{cor:Dk} and computing as in the proof of Theorem~\ref{thm:local_extrema_W2}, it follows that
\begin{align*}
\left.\frac{d^2}{dt^2}\hmW_k(dx^2,\phi_t,\tau)\right|_{t=0} & = \int_{\bR^n} \tau^k\left[ \hT_{k-1,\phi}(\nabla\psi,\nabla\psi) - \frac{1}{2\tau}\hsigma_{k-1,\phi}\psi^2 \right] (4\pi\tau)^{-\frac{n}{2}}e^{-\phi}\dvol \\
& = \left(-\frac{1}{2}\right)^{k-1}\frac{\tau}{(k-1)!} \int_{\bR^n} \left[ \lv\nabla\psi\rv^2 - \frac{1}{2\tau}\psi^2 \right] (4\pi\tau)^{-\frac{n}{2}}e^{-\phi}\dvol ,
\end{align*}
where the second equality uses the assumption $\Ric_\phi=\frac{1}{2\tau}g$.  Since $\int_{\bR^n}\psi e^{-\phi}=0$, it follows immediately that~\eqref{eqn:local_extrema_Wk} holds with equality if and only if $\psi(x)=a\cdot x$ for some fixed $a\in\bR^n$, which is readily seen to be equivalent to the condition that $(\phi_t,\tau)$ is tangent to a curve in $\mC_1(dx^2)$ through shrinking Gaussians.
\end{proof}

\begin{remark}
 \label{rk:stability_steady_expanding}
 Theorem~\ref{thm:local_extrema_W1}, Theorem~\ref{thm:local_extrema_W2} and Theorem~\ref{thm:local_extrema_Wk} all admit analogues for steady and expanding gradient Ricci solitons.  More precisely, if one wishes to study stability for steady (resp.\ expanding) gradient Ricci solitons among the class of manifolds with density over $(M^n,g)$, one should consider the total weighted $\sigma_k$-curvature functional with $\lambda=0$ (resp.\ $\lambda=-1/2\tau<0$); cf.\ \cite{FeldmanIlmanenNi2005,Perelman1}.  Provided the integrals make sense, the arguments given in the shrinking case readily generalize, with appropriate sign changes, to the steady and expanding cases.
\end{remark}

\bibliographystyle{abbrv}
\bibliography{../bib}
\end{document}